\documentclass[12pt,a4paper]{amsart}
\usepackage{latexsym}
\usepackage{amssymb}
\usepackage{amscd}
\usepackage[all]{xy}
\usepackage[cp850]{inputenc}
\usepackage[mathscr]{eucal}
\theoremstyle{plain}
\newtheorem{prob}{Problem}[section]
\newtheorem{defin}[prob]{Definition}
\newtheorem{theorem}[prob]{Theorem}

\newtheorem*{theorem*}{Theorem}
\newtheorem*{question*}{Question}
\newtheorem{prop}[prob]{Proposition}
\newtheorem{corollary}[prob]{Corollary}
\newtheorem{lemma}[prob]{Lemma}
\theoremstyle{remark}

\newtheorem{remark}{Remark}

\newcommand{\R}{\mathbb{R}}
\newcommand{\N}{\mathbb{N}}
\newcommand{\T}{\mathbb{T}}

\newcommand{\C}{\mathbb{C}}
\newcommand{\D}{\mathbb{D}}
\newcommand{\e}{\varepsilon}

\newcommand{\adef}{\begin{defin}}
\newcommand{\zdef}{\end{defin}}

\DeclareMathOperator*{\esssup}{ess\,sup}

\newcommand{\dist}{\textrm{dist}}

\title[Stability of the differential process]{On the stability of the differential process\\
generated by complex interpolation}

\subjclass{Primary: }

\author{Jes\'us M. F. Castillo, Willian H. G. Corr\^ea, Valentin Ferenczi, Manuel Gonz\'alez}

\address{Instituto de Matem\'aticas, Universidad de Extremadura,
Avenida de Elvas s/n, 06011 Badajoz, Spain.} \email{castillo@unex.es}

\address{Departamento de Matem\'atica, Instituto de Matem\'atica e
Estat\'\i stica, Universidade de S\~ao Paulo, rua do Mat\~ao 1010,
05508-090 S\~ao Paulo SP, Brazil} \email{willhans@ime.usp.br}

\address{Departamento de Matem\'atica, Instituto de Matem\'atica e
Estat\'\i stica, Universidade de S\~ao Paulo, rua do Mat\~ao 1010,
05508-090 S\~ao Paulo SP, Brazil.} \email{ferenczi@ime.usp.br}

\address{Departamento de Matem\'aticas, Universidad de Cantabria,
Avenida de los Castros s/n, 39071 Santander, Spain.} \email{manuel.gonzalez@unican.es}

\thanks{The research of the first author was supported in part by Project IB16056 de la
Junta de Extremadura; the research of the first and fourth authors was supported in part
by Project MTM2016-76958, Spain.\newline
The research of the second author was supported in part by CNPq, grant 140413/2016-2,
CAPES, PDSE program 88881.134107/2016-0, and FAPESP, grants 2016/25574-8 and 2018/03765-1.
The research of the third author was supported by FAPESP, grants 2013/11390-4,
2015/17216-1, 2016/25574-8 and by CNPq, grant 303034/2015-7}

\begin{document}

\begin{abstract} We study the stability of the differential process of Rochberg and Weiss associated to an
analytic family of Banach spaces obtained using the complex interpolation method for families.
In the context of K\"othe function spaces we complete earlier results of Kalton (who showed that there is global bounded stability for pairs of K\"othe spaces) by showing that there is global (bounded) stability for families of up to three K\"othe spaces distributed in arcs on the unit sphere while there is no (bounded) stability for families of four or more K\"othe spaces. In the context or arbitrary pairs of Banach spaces we  present local stability results and global isometric stability results. \end{abstract}

\maketitle

\tableofcontents

\section{Introduction}\label{introduction}
Stability problems associated to interpolation processes have been
a central topic in the theory since its inception. Stability issues about the differential process associated to an analytic family of Banach spaces have also been considered since the seminal work of Rochberg and Weiss \cite{derivatives}. In this paper we will start with an interpolation family $(X_w)_{w\in \partial U}$ on the border of an open subset $U$ of $\mathbb{C}$ conformally equivalent to the open unit disc $\mathbb D$ and will study different stability problems connected to the analytic family of Banach spaces $(X_z)_{z\in U}$ obtained by the complex interpolation method of \cite{Coifman1982}.

Starting with a suitable family $(X_w)_{w\in \partial U}$ of Banach spaces, a complex interpolation method consists in constructing a certain Banach space $\mathscr F$ of analytic functions on $U$ with values in a Banach space $\Sigma$ to then obtain the analytic family of Banach spaces $X_z= \{f(z) : f\in \mathscr F\}$, endowed with the quotient norm in $\mathscr F/\ker \delta_z$, where $\delta_z:\mathscr F\to\Sigma$ obviously denotes the
continuous evaluation map at $z\in U$. A particularly important case is the complex method described in \cite{BL}, in which $U$ is fixed as the unit strip $\mathbb S=\{z\in\C : 0<\textrm{Re} z<1\}$ and the starting family is just an interpolation pair $(X_0,X_1)$ of Banach spaces. In this case $X_z=X_{\textrm{Re}\, z}$; so it is usual to consider only the scale $(X_\theta)_{0<\theta<1}$.\smallskip

Analytic families of Banach spaces are also relevant to other topics in Banach space theory such as the construction of
uniformly convex hereditarily indecomposable spaces \cite{ferencziunif}, the study of $\theta$-Hilbertian spaces \cite{pisierlattices} introduced by Pisier and related to a question of V. Lafforgue (see the final comments of Section \ref{sect:Lafforgue}), or problems about the
uniform structure of Banach spaces. Recall that the question of whether the unit sphere of a uniformly convex space is uniformly
homeomorphic to the unit sphere of a Hilbert space can be positively answered in K\"othe spaces using interpolation methods: if $X_0$ and $X_1$ are uniformly convex spaces, then the unit spheres of $X_\theta$
and $X_\nu$ are uniformly homeomorphic for $0<\theta, \nu<1$ by a result of
Daher \cite{daher}; this fact, together with an extrapolation theorem of Pisier \cite{pisierlattices},
implies that the unit sphere of a uniformly convex K\"othe space is uniformly homeomorphic to the unit
sphere of the Hilbert space (see also \cite{chaatit}). Thus, an extrapolation theorem for arbitrary uniformly convex spaces would provide a positive answer to the problem.\medskip

Analytic families of Banach spaces generated by an interpolation process in turn generate a differential process $z\to \Omega_z$ for $z\in U$, where  $\Omega_z$ is a certain non-linear map defined on $X_z$, called the associated derivation at $z$. In the context of K\"othe spaces, derivations are centralizers, in the terminology of Kalton \cite{kaltmem,kaltdiff}, and therefore can be used in the standard way to generate twisted sums \begin{equation}\label{exact}\begin{CD}
0@>>> X_z@>>> d_{\Omega_z}X_z @>>> X_z@>>>0
\end{CD}\end{equation}
Rochberg's approach \cite{roch}, however, contemplates the formation of the so-called derived spaces $dX_z= \{(f'(z),f(z)) : f\in \mathscr F\}$ endowed with the obvious quotient norm to then show that both constructions are isomorphic; i.e., $dX_z \sim d_{\Omega_z}X_z$.\medskip

Thus, the stability of the differential process associated to an analytic family $(X_z)_{z\in U}$ can be studied at several levels.
At the basic level, one considers the stability of isomorphic properties $\mathcal P$ of the spaces
$X_z$ either under small perturbations in the parameter $z$ (local stability) or for the whole range of the
parameter (global stability). Results of this kind have been obtained by many authors. Let us mention one especially interesting obtained by Kalton and Ostrovskii \cite{kaltostr}: If $d_K(A,B)$ denotes the Kadets distance between two Banach spaces $A$ and $B$, a property $\mathcal P$ is said to be open if for every $X$ having $\mathcal P$ there exists $C_X>0$ such that $Y$ has $\mathcal P$ when $d_K(X,Y)<C_X$, while $\mathcal P$ is said to be stable if there exists $C>0$ such that if $X$ has $\mathcal P$ and $d_K(X,Y)<C$ then
$Y$ has $\mathcal P$. Many examples of open and stable properties can be found in \cite{AAG:92} or \cite[Section 5]{kaltostr}. Kalton and Ostrovskii show \cite[Theorem. 4.5]{kaltostr} that $d_K(X_t, X_s) \leq 2{\sf h}(t, s)$, where ${\sf h}$
is the pseudo-hyperbolic distance on $U$ (see Definition \ref{def:hyp-dist}). Thus, at its basic level, the differential process has local stability with respect to open properties and global stability with respect to stable properties.\medskip

At the first level we will consider stability problems for the family $(dX_z)_{z\in U}$ of derived spaces. We will also consider stability problems at level $n$, i.e., stability problems for the families of higher order Rochberg's derived spaces \cite{roch} $d^nX_z= \{(\frac{1}{n!}f^{(n)}(z),\ldots, f(z)) : f\in \mathscr F\}$. These  spaces are endowed with the obvious quotient norm and can also be interpreted as twisted sum spaces \cite{cck}.
As a typical result, we will show a generalized form for the Kalton-Ostrovskii result mentioned
before: $d_K(d^nX_z, d^nX_\eta)\leq 4(n+1){\sf h}(z,\eta)$ , which implies local/global stability for open/stable properties of $d^nX_z$; see Theorem \ref{kadetsderived}.\medskip

The interpretation of derived spaces $dX_z$ as twisted sum spaces (of $X_z$) generated by the corresponding derivation $\Omega_z$ allows
one to study the stability of the exact sequences involved, which is what we will mainly do in the paper. Let us recall that an exact sequence
like (\ref{exact}) is said to split when $X_z$ is complemented in $dX_z$; something that happens when $\Omega_z$ can be written as the sum of a bounded plus a linear map, usually refereed to as: $\Omega_z$ is trivial. Thus, two derivations $\Omega_z$ and $\Omega_z'$ are said to be equivalent when $\Omega_z - \Omega_z'$ is trivial. Kalton's approach to complex
interpolation instead relies on the use of bounded derivations and the notion of bounded equivalence: two derivations $\Omega_z$ and $\Omega_z'$ are said to be boundedly equivalent when $\Omega_z - \Omega_z'$ is bounded. Probably the first stability results at level one have been those obtained by Cwikel, Jawerth, Milman and Rochberg \cite{cjmr} for the minimal $(\theta,1)$-interpolation method
applied to an interpolation pair $(X_0,X_1)$. They reinterpret the results of Zafran \cite{zaf} to show that whenever $\Omega_\theta$ is
bounded for some $0<\theta<1$, then all $\Omega_z$ are bounded and, moreover, $X_0=X_1$ up to a renorming. Kalton
obtains in \cite{kaltdiff} a similar optimal stability result in the context of complex interpolation for pairs of K\"othe
function spaces: $\Omega_\theta$ is bounded for some $\theta\in \mathbb S$ if and only all $\Omega_z$ are bounded for all $z\in \mathbb S$ and, moreover, $X_0 = X_1$, up to an equivalent renorming. See Theorem \ref{kalthm} for the precise statement. Kalton's result leaves several questions unanswered, and a good part of this paper is devoted to solving them. We complete Kalton's result by showing: i) in the context of complex interpolation pairs $(X_0,X_1)$ of superreflexive K\"othe spaces $\Omega_\theta$ is trivial if and only if there is a weight function $w$ so that $X_0=X_1(w)$, up to an equivalent renorming, thus solving the stability problem for splitting (for pairs of K\"othe spaces).
ii) The stability results for pairs remain valid for families of up
to three K\"othe spaces distributed in three arcs of the unit circle $\T$, but fail for
families of four K\"othe spaces. Somehow this marks the limit of validity for Kalton's theorems. If one abandons the
three arcs configuration then it is worth to take into account the results of Qiu  \cite{Yanqi-Qiu}, who shows that at the basic $0$ level complex interpolation for families is stable under rearrangements for two spaces, but it is not stable for three spaces. His results, however, only considers finite-dimensional spaces, while the non-stability we describe concerns isomorphic properties.\medskip

We move then to consider stability problems in the context of couples and families of arbitrary Banach spaces. Regarding families, Theorem \ref{stabilityfam} presents the  $1$-level interpretation of the classical reiteration result for
families of Coifman, Cwikel, Rochberg, Sagher and Weiss \cite[Theorem 5.1]{Coifman1982}; this result explains, to some extent, the lack of stability in the previous counterexamples and can be used to obtain other natural counterexamples. We thank B.\ Maurey and G.\ Pisier at this point for helpful discussions. In the construction of new counterexamples an analogue of Rochberg's concept of flat analytic family \cite{complex} is used. Let $\|\cdot\|$ be a norm on $\C^n$ and let $(T_z)$ be a family of invertible linear maps on
$\C^n$ which vary analytically with $z\in \mathbb D$, the unit disc.
Define $\|x\|_z = \|T_z^{-1} x\|$. The family $(\C^n, \|\cdot\|_z)_{z\in \mathbb D}$ is called
a \emph{flat analytic family} on $\mathbb D$. The transport of this concept to infinite dimensional spaces causes some complications, and we introduce a notion of ``coherence" to handle them in Section \ref{sect:Lafforgue}. Proposition \ref{cbounded} shows the existence of a flat analytic family of K\"othe sequence spaces with norms $\|x\|_z = \|e^{-D(z)}x\|_2$ ($z\in \mathbb D$) generated by an analytic family $D(z)$ of
diagonal operators for which the derivation map $\Omega_z$ is linear and does not depend on $z$.

Regarding couples, curiously, the existence of local or global stability for the differential process associated to complex
interpolation of a couple of Banach spaces remains still an open problem; precisely, \emph{Assume $(X_0,X_1)$ is a pair of Banach spaces such that $\Omega_\theta$ is bounded for some
$0<\theta<1$.  Does it follow that $X_0=X_1$ up to equivalence of norms?}\medskip

In Section \ref{stabilitygen} we present isometric stability results for arbitrary couples of Banach spaces having a common
Schauder basis and for couples of r.i. K\"othe spaces. A key role in our analysis is played by the properties of the extremal functions and by some differential estimates for the norm in an interpolation scale. In \cite[Theorem 5.2]{cjmr} Cwikel et al. obtained the estimate
$$
\frac{d}{d\theta}\|a\|_{\theta,1} \sim \|a\|_{\theta,1} + \|\Omega_\theta a\|_{\theta,1}
$$
for the minimal $(\theta,1)$-method applied to a pair $(X_0, X_1)$ when $X_0$ is continuously embedded in $X_1$. Our version of this estimate for the complex method (Lemma \ref{locallemma}) is
$$
\left|\frac{d}{dt}\|a\|_{t}\right|_{t = \theta^{\pm}} \leq \|\Omega_\theta a\|_{\theta}.
$$
from which we derive a number of stability results for pairs. In many standard situations derivations are uniquely defined
so it makes sense to study exact stability (instead of up to
a bounded or a bounded plus linear perturbation) problems. We show that exact stability is related to {\em isometric} characterizations of $X_0$ and $X_1$. In particular, Theorem \ref{unificado} provides a complete and explicit characterization of pairs $(X_0,X_1)$ of spaces
for which $\Omega_\theta$ is linear.

\section{Preliminary results}\label{firstresults}

For background on the theory of twisted sums and diagrams we refer to \cite{accgmLN,castgonz}.
A twisted sum of two Banach spaces $Y$, $Z$ is  a quasi-Banach space $X$ which has a closed subspace
isomorphic to $Y$ such that the quotient $X/Y$ is isomorphic to $Z$.
An exact sequence $0\to Y\to X\to Z\to 0$
of Banach spaces and continuous operators is a diagram in which the kernel of each arrow coincides with the
image of the preceding one.
Thus, the open mapping theorem yields that the middle space $X$ is a twisted sum of $Y$ and $Z$.
The simplest exact sequence is obtained taking $X= Y\oplus Z$ with embedding $y\to (y,0)$ and quotient map
$(y,z)\to z$.
Two exact sequences $0 \to Y \to X_1 \to Z \to 0$ and $0 \to Y \to X_2 \to Z \to 0$ are said to be
\emph{equivalent} if there exists an operator $T:X_1\to X_2$ such that the following diagram commutes:
$$
\begin{CD}
0 @>>>Y@>>>X_1@>>>Z@>>>0\\
&&@| @VVTV @|\\
0 @>>>Y@>>>X_2@>>>Z@>>>0.
\end{CD}$$

The classical 3-lemma \cite[p. 3]{castgonz} shows that $T$ must be an isomorphism.
An exact sequence is said to be \emph{trivia}l if it is equivalent to $0 \to Y \to Y \oplus Z \to Z \to 0$.
In this case we also say that the exact sequence \emph{splits}.
This is equivalent to the subspace $Y$ being complemented in $X$.
\smallskip

Kalton \cite{kaltmem,kaltdiff} developed a deep theory connecting derivations and twisted sums in
the specific case of K\"othe function spaces that we briefly describe now because it is essential to
understand our work. Let us even present Kalton's definition of K\"othe function space since it is slightly different from the standard one
\cite{lindtzaf1}. According to Kalton \cite[p.482]{kaltdiff}, given a $\sigma$-finite measure Polish space $(S, \mu)$, a norm $\|\cdot\|_X$ on a vector subspace $X$ of $L_0(\mu)$, the space of all complex-valued functions on $S$ is admissible when $X= \{f\in L_0: \|f\|<+\infty\}$ and its closed unit ball is closed in $L_0$ and, moreover, there exist strictly positive functions $h, k\in L_0$ such that
$\|h f\|_1 \leq \|f\|_X \leq \|k f\|_\infty$ for every $f \in L_0$. A K\"othe space is a sublattice of $L_0$ endowed with an admissible norm. Let now $X$ be a K\"othe function space.
A \emph{centralizer on $X$} is a homogeneous map $\Omega: X \to L_0(\mu)$ for which there is a
constant $C$ such that, given $f\in L_\infty(\mu)$ and $x\in X$, $\Omega(fx)- f\Omega(x)\in X$ and
$\|\Omega(fx)- f\Omega(x)\|_X\leq C\|f\|_\infty  \|x\|_X.$

A centralizer $\Omega$ on $X$ induces an exact sequence
$$
\begin{CD}
0@>>> X@>j>> X\oplus_\Omega X  @>q>> X@>>>0,
\end{CD}
$$
where $X\oplus_\Omega X = \{(f, x)\in L_0\times X: f- \Omega x \in X\}$, endowed with the quasi-norm
$\|(f, x)\|_\Omega = \|f - \Omega x\|_X + \|x\|_X$, with inclusion $j(y) = (y,0)$ and quotient map
$q(f,x)=  x$.

We say that a centralizer $\Omega$ is \emph{trivial} if the exact sequence induced by $\Omega$ splits.
We have the following known equivalence of triviality:

\begin{prop}\label{trivial-centr}
A centralizer $\Omega: X \to L_0(\mu)$ is trivial if and only if there exists a linear map
$L: X\to L_0(\mu)$ such that $\Omega - L$  is a bounded map from $X$ to $X$.
\end{prop}
\begin{proof}
If a map $L$ as above exists then the map $(f,x)\to (f-Lx,0)$ is a linear bounded projection on
$X\oplus_\Omega X$ with range $j(X)$.
Indeed, $f-Lx=f-\Omega x+\Omega x- Lx\in X$
and $\|f-Lx\|_X\leq \|(f,x)\|_\Omega+\|(\Omega-L)x\|_X$.

Conversely, if $\Omega$ is trivial then there is a bounded linear map $S:X\to X\oplus_\Omega X$
such that $qS$ is the identity on $X$ \cite[Lemma 1.1.a]{castgonz}.
Then $Sx=(Lx,x)$ for some linear map $L:X\to L_0(\mu)$.
Since $\|(Lx,x)\|_\Omega = \|Lx - \Omega x\|_X + \|x\|_X\leq\|S\|\cdot\|x\|_X$, $L$ satisfies
the required conditions.
\end{proof}

A centralizer $\Omega$ on $X$ is said to be \emph{real} if $\Omega(x)$ is real whenever $x\in X$ is real.
Kalton's theorem stated below establishes that all real centralizers essentially arise from complex
interpolation of an interpolation pair of K\"othe spaces. In the next theorem $\Omega_\theta$ denotes the derivation on $X_\theta$ induced by the interpolation pair $(X_0,X_1)$, which will be fully described in the next section.
%

\begin{theorem}\label{kalthm} \emph{\cite{kaltmem,kaltdiff}}
\begin{enumerate}
\item Given an interpolation pair $(X_0,X_1)$ of complex K\"othe function spaces and $0<\theta<1$,
the derivation $\Omega_\theta$ is a (real) centralizer on $X_\theta$.
\item For every real centralizer $\Omega$ on a separable superreflexive K\"othe function space $X$ there is a
number $\varepsilon>0$ and an interpolation pair $(X_0, X_1)$ of K\"othe function spaces so that $X=X_\theta$
for some $0<\theta<1$ and $\varepsilon\Omega - \Omega_\theta: X_\theta \to X_\theta$ is a bounded map.
\item The derivation $\Omega_\theta$ is bounded as a map $X_\theta \to X_\theta$ for some $\theta$ if and
only if $X_0=X_1$, up to an equivalent renorming.
In this case  $\Omega_\theta$ is bounded for all $\theta$.
\end{enumerate}
\end{theorem}



Recall that given two closed subspaces $M, N$ of a Banach space $Z$, and denoting $B_M$ the unit ball of $M$,
the \emph{gap $g(M,N)$ between $M$ and $N$} is defined as follows:
$$
g(M,N) = \max \big\{\sup_{x\in B_M}\dist(x,B_N), \sup_{y\in B_N}\dist(y,B_M)\big\}.
$$

The \emph{Kadets distance} $d_K(X,Y)$ between two Banach spaces $X$ and $Y$ is the infimum of the gap
$g(i(X), j(Y))$ taken over all the isometric embeddings of $i, j$ of $X, Y$ into a common superspace.
%

\begin{prop}\label{prop:gap-quot} \cite[Theorem 4.1]{kaltostr}
Let $E$ and $F$ be closed subspaces of a Banach space $Z$. Then $d_K(Z/E,Z/F)\leq 2 g(E,F)$.
\end{prop}

\section{Kalton spaces of analytic functions}\label{sect:admissible}

Here we present the abstract version of the complex interpolation method introduced in \cite{kalt-mon}
and other previous papers of Kalton \cite{kaltmem,kaltdiff,kaltostr}.
Along the section $U$ will be an open subset of $\C$ conformally equivalent to the unit disc $\mathbb D$.
The closure and the boundary of $U$ are denoted $\overline U$ and $\partial U$, and we will write
$\T=\partial \mathbb D$.

\begin{defin}\label{def:admissible}
A \emph{Kalton space} is a Banach space $\mathscr{F}\equiv (\mathscr{F}(U, \Sigma),\|\cdot\|_{\mathscr F})$
of analytic functions on $U$ with values in a complex Banach space $\Sigma $ satisfying the following
conditions:
\begin{itemize}
\item[(a)] For each $z\in U$, the evaluation map $\delta_z:\mathscr F\to {\Sigma}$ is bounded.
\item[(b)] If $\varphi:U\to\mathbb D$ is a conformal equivalence and $f:U\to \Sigma$ is an
analytic map, then $f\in\mathscr F$ if and only if $\varphi\cdot f\in\mathscr F$, and in this case
$\|\varphi\cdot f\|_{\mathscr F}= \|f\|_{\mathscr F}$.
\end{itemize}
\end{defin}


Given a Kalton space $\mathscr{F}(U, \Sigma)$, for each $z\in U$ we define
$$
X_z=\{x\in {\Sigma}: x = f(z) \quad \text{ for some } f\in\mathscr F\}
$$
which endowed with the norm $\|x\|_z=\inf\{\|f\|_{\mathscr F}: x = f(z)\}$ is isometric to
$\mathscr F/\ker\delta_z$.

The family $(X_z)_{z\in U}$ is called an \emph{analytic family} of Banach spaces on $U$, and a function
$f_{x,z}\in\mathscr F$ such that $f_{x,z}(z)=x$ and $\|f_{x,z}\|_{\mathscr F}\leq c\|x\|_z$ is called
a $c$-\emph{extremal} (for $x$ at $z$).
\smallskip

There are different ways of generating Kalton spaces, and here is where complex interpolation enters the game.
We shall mainly consider two cases: $U=\mathbb S$, more suitable to handle interpolation pairs \cite{BL},
and $U=\mathbb D$, more suitable for interpolating families \cite{Coifman1982}.
The analytic families of Banach spaces $(X_z)_{z\in U}$ generated via complex interpolation satisfy the
\emph{interpolation property:} whenever an operator $T:\Sigma\to \Sigma$ induces a norm one operator
$T: X_\omega \to X_\omega$ for all $\omega\in \partial U$ then it also induces a continuous operator
$T: X_z\to X_z$ for all $z\in U$ with some control on its norm.
\smallskip

\emph{Complex interpolation for pairs.}
An \emph{interpolation pair} $(X_0, X_1)$ is a pair of Banach spaces, both of them linear and
continuously contained in a bigger Hausdorff topological vector space $\Sigma$ which can be
assumed  to be $\Sigma = X_0+X_1$ endowed with the norm
$\|x\|= \inf \{\|x_0\|_0 + \|x_1\|_1: x= x_0 + x_1\; x_j\in X_j \; \mathrm{for}\; j=0,1\}$.
The pair will be called {\em regular} if, additionally, $\Delta=X_0\cap X_1$ is dense in both $X_0$
and $X_1$.
The space $\Delta$ endowed with the norm $\|x\|_{\Delta} = \max\{\|x\|_{X_0}, \|x\|_{X_1}\}$ is a
Banach space, and the inclusions $\Delta\to X_i\to\Sigma$ are contractions.

The \emph{Calderon space} $\mathscr{C}=\mathscr C(\mathbb{S}, X_0+X_1)$ is formed by those bounded
continuous functions $F:\overline{\mathbb{S}} \to X_0+X_1$ which are analytic on $\mathbb{S}$ and
such that the maps $t \mapsto F(k + ti) \in X_k$ are continuous and bounded, $k=0,1$.
Endowed with the norm $\|F\|_{\mathscr C}=\sup\{\|F(k+ti)\|_{X_k}: t\in\R, k=0,1\}<\infty$,
$\mathscr{C}$ is a Kalton space and the analytic family $(X_z)_{z\in \mathbb S}$ satisfies the
interpolation property.

An alternative description is given in \cite{daher}: Let $P^{\mathbb{S}}$ be the Poisson kernel on
$\partial \mathbb{S}$. Denoting $\overline X = (X_0,X_1)$, we consider the space
$\mathscr{F}^\infty(\overline{X})\equiv\mathscr{F}^\infty(\mathbb{S},\Sigma)$ of all functions
$F: \overline{\mathbb{S}} \to \Sigma$ analytic on $\mathbb{S}$ such that $F(j+it)\in X_j$
for $j=0,1$ and $t\in\mathbb{R}$, the maps $f_j:t\in\mathbb{R}\to F(j+it)\in X_j$ ($j=0,1$) are
Bochner measurable, $F$ has a Poisson representation
\[
F(z) = \int_{\partial \mathbb{S}} F(w) P^{\mathbb{S}}(w) dw
\]
and
$\|F\|_{\mathscr{F}^\infty (\overline{X})} = \max_{j=0,1}\|f_j\|_{L_\infty(\mathbb{R},X_j)}<\infty$.

It is not difficult to check that the space $\mathscr{F}^\infty(\overline{X})$ endowed with the norm
$\|\cdot\|_{\mathscr{F}^\infty(\overline{X})}$ is a Kalton space of analytic functions on $\mathbb{S}$.
Moreover, for $0<\theta<1$, the associated spaces $X_\theta$ coincide (with equality of norms) with the
spaces obtained using the previous description \cite[p. 288]{daher} via the Calderon space $\mathscr C$. The same is true if instead of $\mathscr{F}^{\infty}(\overline{X})$ we use the spaces $\mathscr{F}^p(\overline{X})$ defined similarly, $p \in [1, \infty)$.
\smallskip

\emph{Complex interpolation for families.}
Here we describe the interpolation method in \cite{Coifman1982} with some slight modifications presented
in \cite{correa}.
We consider an \emph{interpolation family} $(X_\omega)_{\omega \in \T}$, for which we assume that
each space $X_w$ is continuously embedded in a Banach space $\Sigma$, the \emph{containing space},
and that there is a subspace $\Delta \subset \cap_{w \in\T} X_w$, the \emph{intersection space},
such that for every $x \in \Delta$ the function $w \mapsto \|x\|_\omega$ is measurable and satisfies
$\int_0^{2\pi} \log^+ \|x\|_{e^{it}} dt < \infty$.
We also suppose that there is a measurable function $k : [0, 2\pi) \rightarrow [0, \infty)$ satisfying
$\int_0^{2\pi} \log^+ k(t) dt < \infty$ and such that $\|x\|_{\Sigma} \leq k(t)\|x\|_{e^{it}}$
for every $x \in \Delta$ and every $t \in [0, 2\pi)$.

We denote by $\mathcal{G}_0$ the space of all analytic functions on $\mathbb{D}$ of the form
$g =\sum_{j=1}^n\psi_j x_j$, with $\psi_j$ in the Smirnov class $N^+$ \cite{duren} and $x_j \in \Delta$,
such that $\|g\| = \esssup_{\omega \in \T} \|g(\omega)\|_\omega < \infty$.
Moreover $\mathcal{G}$ is the completion of $\mathcal{G}_0$.

For each $z_0\in \mathbb{D}$ we define two spaces.
The first one is $X_{\{z_0\}}$, the completion of $\Delta$ with respect to the norm
$\|x\|_{\{z_0\}} = \inf\{\|g\| : g \in \mathcal{G}_0, g(z_0) = x\}$, and the second one is
$X_{[z_0]} = \{f(z_0) : f \in \mathcal{G}\}$ endowed with the natural quotient norm.
By \cite[Proposition 1.5]{correa} $X_{\{z_0\}} = X_{[z_0]}$ isometrically for every $z_0\in \mathbb{D}$
when $\mathcal{G}\equiv \mathcal{G}(\mathbb{D},\Sigma)$ is a Kalton space.
Moreover the associated analytic family $(X_z)_{z\in \mathbb D}$ satisfies the interpolation property.
\smallskip


Given $z\in \mathbb{D}$, the Poisson kernel $P_z(\omega)$ on $\T$ (see \cite[Section 1]{Coifman1982})
provides the harmonic measure $d\mu_z(\omega) =  P_z(\omega) d\omega$ on $\T$,
and each function $\alpha$ on $\T$ which is integrable with respect to $d\mu_z$ can be extended to an
harmonic function on $\overline{\mathbb{D}}$ by the formula:
$$
\alpha(z) = \int_{\T} \alpha(\omega) P_z(\omega) d\omega.
$$
The harmonic conjugate $\tilde{\alpha}$ of $\alpha$ with $\tilde{\alpha}(0) = 0$ is given by
$\tilde{\alpha}(z) = \int_{\T} \alpha(\omega) \tilde{P_z}(\omega) d\omega$, where $\tilde{P_z}(\omega)$
is the conjugate Poisson kernel.
Next we state the reiteration theorem  for later use.

\begin{theorem} \cite[Theorem 5.1]{Coifman1982}\label{interpcoupfam}
Let $(X_0, X_1)$ be an interpolation pair of Banach spaces, let $\alpha:\T\to [0,1]$ be a
measurable function, and let $X_\omega = (X_0, X_1)_{\alpha(\omega)}$  for $\omega\in \T$.
Then $\{X_\omega\}_{\omega\in \T}$ is an interpolation family and $X_{[z]} = (X_0, X_1)_{\alpha(z)}$
for each $z\in \mathbb D$, with equality of norms.
Moreover, if both $\inf_{\omega\in\T} \alpha(\omega)$ and $\sup_{\omega\in\T} \alpha(\omega)$
are attained, then $X_{\{z\}} = X_{[z]}$.
\end{theorem}
\smallskip

\emph{Complex interpolation for admissible families of K\"othe spaces.}
In \cite{kaltdiff} Kalton considers a variation of the complex interpolation method in \cite{Coifman1982}
for families of K\"othe function spaces of $\mu$-measurable functions, where $\mu$ is a $\sigma$-finite
Borel measure on a Polish space.
For $U=\mathbb D$, he defines the notion of \emph{admissible family} of K\"othe function
spaces $\{X_\omega\}_{\omega \in \T}$, for which there exist two strictly positive $h,k\in L_0(\mu)$ such
that given $x\in L_0(\mu)$, we have $\|xh\|_1\leq \|x\|_\omega\leq \|xk\|_\infty$ for every
$w\in \T$.
The family is \emph{strongly admissible} if, additionally, there exists a countable dimensional
subspace $V$ of $L_0(\mu)$ such that $V \cap B_{X_\omega}$ is $L_0(\mu)$-dense in $B_{X_\omega}$
for a.\ e.\ $\omega \in \T$.
These conditions hold in most reasonable situations. For example, if the family is finite then it is strongly admissible.
We refer to \cite{kaltdiff} for the details.

Given an admissible family $\{X_\omega\}_{\omega \in \T}$ of K\"othe spaces, the role of Kalton space is
played by the space $\mathcal{N}^+(\mathbb D)$ of functions $f : \mathbb D \rightarrow L_0(\mu)$ such that
\begin{itemize}
\item for $\mu$-almost every $s$, the function $F_s : \mathbb{D} \rightarrow \mathbb{C}$ defined
by $F_s(z) = f(z)(s)$ belongs to the Smirnov class $N^+$ (see Section \ref{sect:admissible}) for every
$z \in \mathbb{D}$;
\item $\|f\| = \esssup_{w \in \mathbb{T}} \|f(w)\|_w < \infty$, where $f(w)$ is the radial limit, in the
$L_0(\mu)$ topology, of $f(z)$ with $z \rightarrow w$ (which exists by Fubini's theorem).
\end{itemize}

The definition of $\mathcal{N}^+(\mathbb D)$ in \cite{kaltdiff} does not include the condition
$\|f\| = \esssup_{w \in \mathbb{T}} \|f(w)\|_w < \infty$,
%
%
but we will need it.
This amendment is harmless since \cite[Proposition 2.4]{kaltdiff} asserts the existence of extremals
in our space, which means that the new space $\mathcal{N}^+(\mathbb D)$ yields the same spaces $X_z$.

\begin{remark}
By \cite[Lemma 2.2]{kaltdiff}, each $f\in \mathcal{N}^+(\mathbb D)$ belongs to the Hardy space
$H^1(L_1(hd\mu))$, hence $\mathcal{N}^+(\mathbb D)$ consists of analytic functions
$f: \mathbb D \to L_1(hd\mu)$, the norms of the evaluation maps are at most 1, and multiplying by a
conformal map is an isometric map.
Moreover the arguments in \cite{Coifman1982} allow us to show that $\mathcal{N}^+(\mathbb D)$ is closed
in $H^1(L_1(hd\mu))$.
Thus $\mathcal{N}^+(\mathbb D)$ satisfies the conditions in Definition \ref{def:admissible} with
$\Sigma=L_1(hd\mu)$.
\end{remark}

\begin{remark}
This method may be transposed from the disk to the strip by means of a conformal map, and it agrees with the method of complex interpolation of couples when we are dealing with a family of two K\"othe spaces distributed in two arcs.
\end{remark}


We also need to recall from \cite{kaltdiff} the notions of \emph{semi-ideal} and \emph{indicator function}.

\begin{defin}
A \emph{semi-ideal} is a cone $\mathcal{I} \subset L_1^+$ such that $g \in \mathcal{I}$ and $0 \leq f \leq g$
imply $f \in \mathcal{I}$.
A \emph{strict semi-ideal} is a semi-ideal which contains a strictly positive element.
\end{defin}

Given a K\"othe function space $X$, we consider the semi-ideal $\mathcal{I}_X$ of all $f \in L_1^+$
such that
$\sup_{x \in B_X} \int f \log^+ \left|x\right| d\mu < \infty$ and
there is $x \in B_X$ such that $\int f\left|\log\left|x\right|\right| d\mu < \infty$.
\smallskip

The \emph{indicator of $X$} is the map  $\Phi_X: \mathcal{I}_X \to \R$ given by
$\Phi_X(f) = \sup\limits_{x \in B_X} \int_{S} f \log \left|x\right| d\mu$.
\smallskip

We will need the following result:

\begin{theorem} \cite[Theorem 4.7]{kaltdiff}\label{thm-interpolated-indicator}
Given a strongly admissible family $\{X_\omega\}_{\omega \in \T}$, there is a strict semi-ideal
$\mathcal{I}$ such that for each $z_0 \in \mathbb{D}$ and $f \in \mathcal{I}$, we have
$\mathcal{I} \subset \mathcal{I}_{X_{z_0}}$, the map $t \mapsto \Phi_{X_{e^{it}}}(f)$ is a bounded
and measurable, and
$$
\Phi_{X_{z_0}}(f) = \frac{1}{2\pi} \int_{0}^{2\pi} \Phi_{X_{e^{it}}}(f) P_{z_0}(e^{it}) dt.
$$
\end{theorem}

The core of Kalton's method is that centralizers on a separable K\"othe space $X$ actually live
on $L_1(\mu)$. More precisely, given a centralizer $\Omega$ on $X$, then $L_1 = X X^*$ by
Lozanovskii's factorization \cite{loza}.
Thus each $f\in L_1$ can be written as $f=x x^*$ with $\|x\|\|x^*\|\leq 2\|f\|$, and one can set
$$
\Omega^{[1]}(f)= \Omega(x)x^*.
$$
This is a centralizer on $L_1$ that, whenever $f= y y^*$ with $y\in X, y^*\in X^*$, it satisfies
$$
\|\Omega^{[1]}(f) - \Omega(y)y^*\|_{L_1} \leq C \|y\|\|y^*\|
$$
for some uniform constant $C>0$. See \cite[Theorem 5.1]{kaltmem} for details.
When $X=L_p$ ($1<p<\infty$), $\Omega^{[1]}(f) = u| f |^{1/q}\Omega(| f |^{1/p})$, where $u|f|$ is
the polar decomposition of $f$ and $p^{-1} + q^{-1}=1$.
\smallskip

Given a centralizer $\Omega$ on a K\"othe space $X$, Kalton considers the strict ideal
$\mathcal{I}_{\Omega} \subset L_1$ of those elements $f\in L_1$ for which $\Omega^{[1]}(f)\in L_1$,
and define on $\mathcal{I}_{\Omega}$ the functional
$$
\Phi^\Omega (f) = \int \Omega^{[1]}(f) d\mu.
$$
The crucial properties of this functional are established in the next result:

\begin{theorem}\label{thm-indicator-of-centralizer} \cite[Proposition 7.4]{kaltdiff}
\begin{enumerate}
\item Let $(X_0, X_1)$ be an interpolation couple of K\"othe spaces and let $\Omega_\theta$ be
the derivation map associated with $X_\theta$.
Then on a suitable semi-ideal one has $\Phi^{\Omega_\theta} = \Phi_{X_0} - \Phi_{X_1}$.
\item Let $\{X_\omega\}_{\omega \in \T}$ be a strongly admissible family.
If $\Omega$ is the centralizer associated to $X_z$ for $z=0$, then on a suitable strict semi-ideal
$\mathcal{I} \subset \mathcal{I}_{\Omega}$ one has that for every $f \in \mathcal{I}$
$$
\Phi^{\Omega}(f) = \frac{1}{2\pi} \int_{-\pi}^{\pi} e^{-it} \Phi_{X_{e^{it}}}(f) dt.
$$
\end{enumerate}
\end{theorem}

\subsection{Derivations, centralizers and twisted sums}

Given a Kalton space $\mathscr{F}(U, \Sigma)$ and $z\in U$, the evaluation map
$\delta'_z:f\in\mathscr F\to f'(z)\in\Sigma $ of the derivative at $z$ is bounded for all $z\in U$
(see Lemma \ref{deltas-lemma} for a precise estimate of its norm).
We also need the following well-known fact, for which we present a proof for the sake of later use.

\begin{prop}\label{aux-prop}
For each $z\in U$, the map $\delta_z'$ is continuous and surjective from $\ker \delta_z$ to $X_z$.
\end{prop}
\begin{proof}
Let $\varphi:U\to \mathbb{D}$ be a conformal equivalence such that $\varphi(z)=0$.
Each $g\in \ker \delta_z$ can be written as $g=\varphi\cdot f$ for some $f\in \mathscr{F}$,
and $g'(z)=\varphi'(z) f(z)\in X_z$, thus $\delta_z'(\ker \delta_z)\subset X_z$ and the continuity
into $X_z$ follows from the closed graph theorem.
Moreover, given $x\in X_z$ and $f\in \mathscr{F}$ with $f(z)=x$,
$g=\varphi(z)^{-1}\varphi\cdot f\in \ker \delta_z$ and $\varphi'(z)=x$, hence
$\delta_z'(\ker \delta_z)= X_z$.
\end{proof}

For each $z\in U$ we consider the space $dX_z = \{(f'(z), f(z)): f\in \mathscr F\}.$
The map $\Delta_z: \mathscr F \to \mathscr{F}$ given by  $\Delta_z(f)= (f'(z), f(z))$ is bounded and
thus $dX_z$ can be endowed with the quotient norm
$\|(a,b)\| = \inf \{\|f\|_{\mathscr F}: f\in \mathscr F, \; f'(z)=a,\; f(z)=b\}$.
The space $dX_z$ admits an exact sequence $0 \to  X_z \to dX_z \to X_z \to 0$ with inclusion
$j_z(x)=(x,0)$ (thanks to Proposition \ref{aux-prop}) and quotient map $q_z(y,x)=x$.
All this yields a commutative diagram:
\begin{equation}\label{full diagram}
\begin{CD}
0 @>>> \ker \delta_z @>>> \mathscr F @>{\delta_z}>> X_z @>>>0\\
&&@V{\delta_z'}VV @VV{\Delta_z}V @|\\
0 @>>> X_z @>>j_z> dX_z @>>q_z> X_z @>>>0
\end{CD}
\end{equation}

Thus we have a method to obtain twisted sums of spaces $X_z$ obtained from a Kalton space $\mathscr{F}$.
The twisted sum space can be described using the so-called \emph{derivation map} given by
$\Omega_z = \delta_z' B_z$, where $B_z: X_z\to \mathscr{F}$ is a homogeneous bounded selection for the
evaluation map $\delta_z:\mathscr{F}\to \Sigma$.
We consider the space
$$
d_{\Omega_z}X_z = \{(y,x)\in \Sigma\times X_z: y- \Omega_zx \in X_z\}
$$
endowed with the quasi-norm $\|(y,x)\| = \|y - \Omega_z x\|_z + \|x\|_z$ so that one has an exact
sequence $0 \to  X_z \to d_{\Omega_z}X_z \to X_z \to 0$ with inclusion $x \to (x,0)$ and quotient
map $(y,x)\to x$.
It is not hard to check \cite{ccs} that this exact sequence is equivalent to the lower row of
(\ref{full diagram}).
Note that different choices of selection $B_z$ lead to different derivations $\Omega_z$, but the
difference between two of these derivations is always a bounded map, so both choices produce isomorphic
derived spaces and equivalent twisted sums.

The derivation map $\Omega_z$ is said to be \emph{trivial} if the associated exact sequence splits.
With the proof of Proposition \ref{trivial-centr} we obtain the following result:

\begin{prop}\label{trivial-deriv}
The derivation map $\Omega_z$ is trivial if and only if there is a linear map $L: X_z\to \Sigma$
such that $\Omega_z - L$  is a bounded map from $X_z$ to $X_z$.
\end{prop}

%

\subsection{Distances and isomorphisms}

It is not difficult to translate results (at levels $0$ and $1$) from the open unit disk $\mathbb D$
to the open unit strip $\mathbb{S}$, and conversely.
Indeed, if $\varphi :\mathbb{S} \rightarrow \mathbb{D}$ is a conformal map and
$(X_\omega)_{\omega\in\mathbb{D}}$ is an interpolation family on $\mathbb D$, then
$Y_{z} = X_{\varphi(z)}$ provides an interpolation family $(Y_z)_{z \in \mathbb{S}}$ on $\mathbb S$.
The corresponding derivation maps are related as follows:
$$
\Omega^{\mathbb{S}}_{z} = \varphi'(z)\Omega^{\mathbb D}_{\varphi(z)}.
$$

Given $s\in U$, we denote by $\varphi_s:U\to \mathbb D$ a conformal equivalence taking $s$ to $0$.
In the case $U= \mathbb{S}$ an example is given by
\begin{equation}
\label{conformal}
\varphi_s(z) = \frac{\sin\left(\pi(z-s)/2\right)}{\sin\left(\pi(z+s)/2\right)}\quad
\textrm{($z\in \mathbb{S}$)}
\end{equation}
for which $\varphi_s'(s)=\pi/(2\sin\pi s)$.
The conformal equivalence $\varphi_s$ is unique up to a multiplicative constant: any other conformal
equivalence $\psi_s$ taking $s$ to $0$ can be written as $\psi_s=f\circ \varphi_s$, where
$f(z)=e^{i \theta}z$ \cite[13.14 Lemma]{bak-newman}.

Given $\mathscr{F}(U, \Sigma)$ and $z\in U$, we denote by $\delta_z^{n}:\mathscr{F}\to \Sigma$
the evaluation of the $n$-{th} derivative at $z$.
We will need the following estimates:

\begin{lemma}\label{deltas-lemma}
Let $\mathscr{F}(U,\Sigma)$ be a Kalton space, $s\in U$ and $n \in \N$. Then
\begin{enumerate}
\item $\|\delta^{n}_s: \mathscr{F} \to \Sigma\| \leq  n! /\dist(s,\partial U)^n.$
\smallskip

\item $\|\delta'_s: \ker \delta_s \to X_s \| =
\inf \{\|\delta'_s x\| : x\in \ker \delta_s, \dist(x, \ker \delta'_s)=1\}= |\varphi_s'(s)|$.
\end{enumerate}
\end{lemma}
\begin{proof}
Given a positively oriented closed rectifiable curve $\Gamma$ in $U$ for which $z$ belongs
to the inside of $\Gamma$, the Cauchy integral formula \cite[Appendix A3]{LaursenNeumann:00}
establishes that, for each $n\in\N_0$,
$$
f^{(n)}(z) = \frac{n!}{2\pi i}\int_\Gamma \frac{f(\omega)}{(w-z)^{n+1}}dw.
$$
We take a number $r$ with $0<r<\dist(s,\partial U)$ and denote by $\Gamma$ the boundary of
the open disc $\mathbb{D}(s,r)$.
By the Cauchy integral formula
$$
\|f^{(n)}(s)\| \leq \frac{n!}{2\pi}\int_\Gamma \frac{\|f(\omega)\|}{r^{n+1}}d|w| \leq \frac{n!}{r^n}
\|f\|_\mathscr{F},
$$
and since we can take $r$ arbitrarily close to $\mathrm{dist}(s,\partial U)$, we get estimate (1).
\smallskip

(2) Clearly $\|\delta'_s: \ker \delta_s \to X_s \| \geq
\inf \{\|\delta'_s x\| : x\in \ker \delta_s, \dist(x, \ker\delta'_s)=1\}$, and given $g\in \ker \delta_s$
the function $f(z) =  \varphi_s'(s)\cdot\varphi_s(z)^{-1}g(z)$ is in
$\mathscr{F}$ and satisfies $f(s)= g'(s)$ and $\|f\|=|\varphi_s'(s)|\|g\|$.
Therefore
$$
\|\delta'_s g\|_s = \|f(s)\|_s \leq |\varphi_s'(s)|\|g\|,
$$
and we get  $\|\delta'_s: \ker \delta_s \to X_s \|\leq |\varphi_s'(s)|$.

Also, given $x\in B_{X_s}$ and $\varepsilon>0$, we can take
$f\in \mathscr{F}$ with $\|f\|<(1+\varepsilon)$ and $f(s)=x$.
Then $g(z) =  \varphi_s'(s)^{-1}\varphi_s(z)\cdot f(z)$ defines  $g\in \ker \delta_s$ with
$\|g\|<(1+\varepsilon)/|\varphi_s'(s)|$ and $g'(s)=x$.
Hence
$$
\delta'_s (B_{\ker \delta_s})\supset |\varphi_s'(s)|(1+\varepsilon)^{-1}B_{X_s},
$$
and we get $\inf \{\|\delta'_s x\| : x\in \ker\delta_s, \dist(x,\ker\delta'_s)=1\}\geq |\varphi_s'(s)|$
finishing the proof.
\end{proof}


Part (2) of Lemma \ref{deltas-lemma} says that $\delta_s': \ker \delta_s \to X_s$ is not only surjective, but
a multiple of a quotient map: the induced injective map $\ker \delta_s/(\ker \delta_s'\cap\ker \delta_s) \to X_s$
is $|\varphi_s'(s)|$ times an isometry.

\begin{lemma}\label{stable-lemma}
For each $f\in \mathscr F$ and $s\in U$, we have $\Omega_s(f(s)) - f'(s)\in X_s$ with
$$
\|\Omega_s(f(s)) - f'(s)\|_s \leq  2\| \delta_s': \ker \delta_s \to X_s\|\|f\|
\leq  2\|f\|/\mathrm{dist} (s, \partial U).
$$
\end{lemma}
\begin{proof}
From $\Omega_s(f(s)) - f'(s)= \delta_s' \left( B_s(f(s)) - f \right)$ with
$B_s(f(s)) - f\in \ker \delta_s$, we get the first part.
For the rest, note that the operator $\delta_s' : \ker \delta_s \to X_s$ is bounded by
Lemma \ref{deltas-lemma}.
\end{proof}

\begin{prop}\label{isomorphic kernels}
Let $s,t\in U$.
\begin{enumerate}
\item The spaces $\ker \delta_s  $ and $\mathscr F$ are isometric.
Consequently, $\ker \delta_s $ and $\ker \delta_t$ are isometric.
\item For every $n\in \N$, $\cap_{0\leq k\leq n} \ker \delta_s^k$ and $\mathscr F$ are isometric.
%
\end{enumerate}
\end{prop}
\begin{proof}
The operator $d_s: \mathscr F\to \ker \delta_s$ given by $d_s(f)(z) = f(z)\varphi_s(z)$
is clearly well-defined and injective, and it is surjective because each $g\in \ker \delta_s$ can
be written as $g = \varphi_s\cdot f$ with $f\in \mathscr F$.

To prove (2), just note that
$(d_s)^{n+1}: \mathscr F \to \bigcap_{0\leq k\leq n} \ker \delta_s^k$
is also an isometry.
\end{proof}

Let $s,t\in U$.
The map $\varphi_s \cdot f \in\ker \delta_s\to \varphi_t \cdot f\in\ker \delta_t$ is a bijective
isometry, but we need a more precise description.
Note that the map $\varphi_{s,t}:U\to \mathbb{D}$ defined by
$$
\varphi_{s,t}(z) =\frac{\varphi_s(z) - \varphi_s(t)}{1 -\overline{\varphi_s(t)}\varphi_s(z)}\quad
\textrm{($z\in U$)}
$$
is a conformal equivalence  satisfying $\varphi_{s,t}(t)=0$.
Moreover, denoting $\alpha=\varphi_s(t)\in \mathbb{D}$, one has
\begin{eqnarray*}
\|\varphi_s-\varphi_{s,t}\|_\infty &=&\sup_{z\in U} \left| \varphi_s(z)-\varphi_{s,t}(z)\right|=
 \sup_{\lambda \in\mathbb{D}} \left|\lambda- \frac{\lambda - \alpha}{1 -\overline{\alpha}\lambda}\right|\\
 &=& \sup_{\omega\in\mathbb{T}} \left|\omega- \frac{\omega - \alpha}{1 -\overline{\alpha}\omega}\right|=
  \sup_{\omega\in\mathbb{T}} \left|\frac{\alpha -\overline{\alpha}\omega^2}{1 -\overline{\alpha}\omega}\right|\\
  &=&  \sup_{\omega\in\mathbb{T}} \left|\frac{\alpha\overline{\omega} -\overline{\alpha}\omega}{\overline{\omega}
   -\overline{\alpha}}\right|\leq 2|\alpha|,
\end{eqnarray*}
since
$|\alpha \overline{\omega} -\overline{\alpha}\omega|\leq |\alpha \overline{\omega}-\alpha \overline{\alpha}|
+ |\alpha \overline{\alpha} - \overline{\alpha}\omega| = 2|\alpha||\overline{\omega} -\overline{\alpha}|$.


\begin{defin}\label{def:hyp-dist} \emph{(\cite{kaltostr})}
The \emph{pseudo-hyperbolic distance} ${\sf h}(\cdot, \cdot)$ on $U$ is defined by
${\sf h}(s,t) = |\varphi_s(t)|$.
\end{defin}

This yields:

\begin{prop}\label{deltas}
For each $n\in\N\cup\{0\}$,
$g\left(\bigcap_{0\leq k\leq n} \ker\delta_s^k, \bigcap_{0\leq k\leq n} \ker\delta_t^k) \right)\leq
2(n+1){\sf h}(s,t)$.
\end{prop}
\begin{proof}
We proceed inductively on $n$. For $n=0$, we take a norm-one $\varphi_s \cdot f\in\ker \delta_s$.
Since $\varphi_{s,t} \cdot f\in\ker \delta_t$ is norm-one and
$\|\varphi_s \cdot f-\varphi_{s,t} \cdot f \|_\mathscr{F} =
\|\varphi_s-\varphi_{s,t}\|_\infty\leq 2{\sf h}(s,t)$,
and we can proceed similarly for each norm-one $\varphi_t \cdot f\in\ker \delta_t$, we get
$g\left( \ker\delta_s, \ker\delta_t \right)\leq 2{\sf h}(s,t)$.

Moreover if the estimate holds for $n-1$ then it also holds for $n$ because
\smallskip

\quad $a^{n+1} - b^{n+1} = a^{n+1} - a^n b + a^n b - b^{n+1} = a^n(a-b)  + (a^n -b^n) b.$
\end{proof}

Since $\mathscr F/\bigcap_{0\leq k\leq n} \ker \delta_s^k = d^nX_s$, Propositions \ref{prop:gap-quot}
and \ref{deltas} provide the following result:

\begin{theorem}\label{kadetsderived}
Given $s,t\in U$ and $n\in\N\cup \{0\}$, $d_K(d^nX_s, d^nX_t)\leq 4(n+1){\sf h}(s,t)$.
\end{theorem}


\begin{corollary}
Let $\mathcal P$ be an open (resp. stable) property.
Assume that there is $s\in U$ so that $d^n X_s$ has $\mathcal P$.
Then $d^n X_t$ has $\mathcal P$ for all $t\in U$ (resp. for all $t$ in an open disc centered in $s$).
\end{corollary}

\subsection{Bounded stability}

Let $\mathscr{F}(U,\Sigma)$ be a Kalton space and let $z\in U$.
Then the exact sequence $0\to X_z\to dX_z\to X_z\to 0$ associated to $\Omega_z:X_z\to \Sigma$
splits if and only if there exists a linear map $L:X_z\to \Sigma$ such that $\Omega-L$ takes $X_z$
to $X_z$ and it is bounded (Proposition \ref{trivial-deriv}).
Kalton's work justifies the importance of the case $\Omega_z$ bounded in interpolation affairs.
Let us accordingly introduce a few related notions.

\adef\label{def:bounded}
The derivation $\Omega_z$ is {\em bounded} when it takes values in $X_z$ and it is bounded as
a map from $X_z$ to $X_z$.
In this case we will say that the induced exact sequence \emph{boundedly splits.}
\zdef

Bounded splitting admits the following characterizations.

\begin{theorem}\label{main}
Let $\mathscr{F}(U,\Sigma)$ be a Kalton space and let $s \in U$.
The following assertions are equivalent:
\begin{enumerate}
\item $\delta_s : \ker \delta_s'\to X_s$ is surjective.
\item $\mathscr{F} = \ker \delta_s + \ker \delta_s'$.
\item There exists $M>0$ such that each $f \in \mathscr{F}$ can be written as $f=g+h$ with
$g \in \ker \delta_s$, $h \in \ker \delta_s'$ and
$\max\{\|g\|_\mathscr{F},\|h\|_\mathscr{F}\} \leq M\|f\|_\mathscr{F}$.
\item $\delta_s'(\mathscr{F})\subset X_s$.
\item $\delta_s': \mathscr{F} \to X_s$ is bounded.
\item $\Omega_s(X_s) \subset X_s$.
\item $\Omega_s: X_s \to X_s$ is bounded.
\end{enumerate}
\end{theorem}

\begin{proof}
Clearly $(1) \Leftarrow (2) \Leftarrow (3)$, $(4) \Leftarrow (5)$ and $(6) \Leftarrow (7)$.
Moreover $(4) \Leftrightarrow (6)$ and $(5) \Leftrightarrow (7)$ follow from Lemma \ref{stable-lemma}.
We will prove $(1) \Rightarrow (3)\Rightarrow (5)$ and $(4) \Rightarrow (2)$.
\smallskip

$(1) \Rightarrow (3)$: Let $f \in \mathscr F$ with $\|f\|=1$.
Since $\delta_s: \ker \delta_s' \to X_s$ is surjective, it is open.
So there exists $r>0$ such that we can find $h\in \ker \delta_s'$ with $\|h\|\leq r\|f(s)\|_s$ and $h(s)=f(s)$.
Since $\|f(s)\|_s\leq\|f\|$, taking $g=f-h\in\ker \delta_s$ we obtain (3) with $M=r+1$.
\smallskip

$(3) \Rightarrow (5)$: Let  $f \in \mathscr F$.
We can be write $f=g+h$ with $g \in \ker \delta_s$, $h \in \ker \delta_s'$ and
$\|g\|_\mathscr{F} \leq M\|f\|_\mathscr{F}$.
Then
$$
\|\delta'_s(f)\|_s = \|\delta'_s(g)\|_s \leq  \| \delta_s': \ker \delta_s \to X_s\|\cdot\|g\|_\mathscr{F}
\leq  M\| \delta_s': \ker \delta_s \to X_s\|\cdot\|f\|_\mathscr{F}.
$$

$(4) \Rightarrow (2)$: We know that the operator $\delta_s': \ker \delta_s \to X_s$ is surjective.
So taking a linear selection $\ell: X_s \to \ker \delta_s$ for $\delta_s'$,
for each $f\in \mathscr F$, $\ell(f'(s))\in \ker \delta_s$ and $f -\ell(f'(s))\in \ker \delta_s'$.
\end{proof}

Condition (6) shows that the requirements in Definition \ref{def:bounded} are redundant.
Condition (2) in Theorem \ref{main} provides a neat description of how the twisted sum space
$d_{\Omega_s} X_s$ splits when $\Omega_s$ is bounded.
Indeed, since $d_{\Omega_s} X_s = \mathscr F/(\ker \delta_s \cap \ker \delta'_s)$ and the subspace $X_s$
embeds in $d_{\Omega_s} X_s$ as $\ker \delta_s/(\ker \delta_s \cap \ker \delta'_s)$, condition (2) gives
$$
\frac{\mathscr F}{\ker \delta_s \cap \ker \delta'_s}=
\frac{\ker \delta_s + \ker \delta_s'}{\ker \delta_s \cap \ker \delta'_s}=
\frac{\ker \delta_s}{\ker \delta_s\cap \ker \delta'_s}\oplus\frac{\ker \delta_s'}{\ker\delta_s \cap \ker\delta'_s}.
$$

Next we define the notions of stability we will study:

\adef
Given a Kalton space $\mathscr F(U, \Sigma)$, the map $z\to \Omega_z$  will be called
the \emph{differential process} associated to the analytic family $(X_z)_{z\in U}$.

Moreover, we will say that this differential process:
\begin{enumerate}
\item has \emph{local stability} if whenever $\Omega_{z_0}$ is trivial then there is $\e>0$
such that $\Omega_z $ is trivial for $|z-z_0|<\e$.
\item has \emph{local bounded stability} if whenever $\Omega_{z_0}$ is bounded then there is $\e>0$
such that $\Omega_z$ is bounded for for $|z-z_0|<\e$.
\item has \emph{global stability} if whenever $\Omega_{z_0}$ is trivial then $\Omega_z$ is trivial
for all $z\in U$.
\item has \emph{global bounded stability} if whenever $\Omega_{z_0}$ is bounded then $\Omega_z$
is bounded for all $z\in U$.
\end{enumerate}
\zdef

\section{Stability of splitting for K\"othe function spaces}\label{Kothe}

Theorem \ref{kalthm} shows that when an analytic family is generated by an interpolation
pair $(X_0,X_1)$ of K\"othe spaces then the differential process is ``rigid", in the sense that
whenever $\Omega_{z_0}$ is bounded at some point $z_0$ then $X_0=Y_0$, up to some equivalent
renorming.
Here we will prove that:
 \begin{itemize}
 \item The differential process associated to families of up to three K\"othe spaces distributed in arcs enjoys global
 (bounded) stability; in fact, it is ``rigid" in the case of bounded stability and ``rigid" up to
 weighted versions in  the case of stability. This can be found in Section \ref{three}.
\item The differential process associated to families of four spaces can fail local bounded stability
(Proposition \ref{cbounded}) or local stability (Proposition \ref{singfam}).
\end{itemize}

\subsection{Stability for pairs of K\"othe spaces}
After Kalton's bounded stability theorem (Theorem \ref{kalthm}), it is a reasonable guess that ``nontrivial
scales" of K\"othe spaces correspond to ``nontrivial centralizers".
The difficulty is that the non-triviality notion involves uncontrolled linear maps, as we can see in
Proposition \ref{trivial-deriv}.
Thus, while Kalton shows \cite{kaltdiff} that the centralizer $\Omega_\theta$ associated to the scale
$(X_0, X_1)_\theta$ of K\"othe function spaces is bounded if and only if $X_0 = X_1$ up to equivalence
of norms, the following question remained open:
\emph{Does the triviality of $\Omega_\theta$ imply that $X_0$ and $X_1$ are equal, or at least isomorphic?}
\smallskip

We shall now prove global stability for pairs of K\"othe spaces.
The following sentence in \cite[p. 364]{casskalt} clearly suggests that it was known to Kalton, at least
in the domain of K\"othe sequence spaces:
\emph{If $(Z_0, Z_1)$ are two super-reflexive sequence spaces and $Z_\theta = [Z_0, Z_1]_\theta$ for
$0<\theta<1$ is the usual interpolation space by the Calderon method, one can define a derivative
$dZ_\theta$ which is a twisted sum $Z_\theta \oplus_{\Omega} Z_\theta$ which splits if and only if
$Z_1 = wZ_0$ for some weight sequence $w = (w(n))$ where $w(n) > 0$ for all n.
These remarks follow easily from the methods of} \cite{kaltdiff}.
\smallskip

Next we recall Kalton's formula \cite[(3.2)]{kaltdiff} for the centralizer $\Omega_\theta$
corresponding to a couple of K\"othe function spaces $(X_0,X_1)$ and $0<\theta<1$.
%
It is well known \cite{calderon} that $X_\theta$ coincides with the space $X_0^{1-\theta} X_1^\theta$,
with
$$
\|x\|_\theta=\inf \{\|y\|_0^{1-\theta}\|z\|_1^{\theta}: y\in X_0, z\in X_1,
|x|=|y|^{1-\theta}|z|^{\theta}\}.
$$
We fix $c>1$. For each $x\in X$ we write $|x|=|a_0(x)|^{1-\theta}|a_1(x)|^\theta$ with
$\|a_0(x)\|_0, \|a_1(x)\|_1 \leq c\|x\|_\theta$, where $a_0$ and $a_1$ are chosen homogeneously.

Then $B_\theta(x)(z) = (\textrm{sgn} x)|a_0(x)|^{1-z} |a_1(x)|^z$ gives an extremal for $x$
at $\theta$, and we obtain
\begin{equation}\label{kpgeneralized}
\Omega_\theta(x) = \delta_\theta'B_\theta(x)  = x\, \log \frac{|a_1(x)|}{|a_0(x)|}.
\end{equation}

Given a K\"othe function space $X$ of $\mu$-measurable functions, a \emph{weight} $w$ is a positive
function in $L_0(\mu)$.
We denote by $X(w)$ the space of all measurable scalar functions $f$ such that $wf\in X$, endowed
with the norm $\|x\|_{w} = \|wx\|_X$.

From the approach in \cite{cfg} we get the following general version of a well-known result for scales of
$L_p$-spaces \cite[5.4.1. Theorem]{BL}:

\begin{prop}\label{twist-weights}
Let $X$ be a K\"othe function space with the Radon-Nikodym property, and let $w_0, w_1$ be two weights.
Then $(X(w_0), X(w_1))_\theta = X(w_0^{1-\theta} w_1^\theta)$ for $0<\theta<1$, with associated
linear centralizer $\Omega_\theta(x) = \log (w_0/w_1)\cdot x$ for $x\in \varphi(X)$.
\end{prop}
\begin{proof}
By \cite[Theorem 4.6]{kalt-mon}, the space $(X(w_0), X(w_1))_\theta$ is isometric to the space
$X(w_0)^{1-\theta} X(w_1)^\theta$ endowed with the norm
\begin{align*}
\|x\|_{\theta}
&= \inf \{ \|a\|_{w_0}^{1-\theta}\|b\|_{w_1}^\theta : a\in X(w_0), b\in X(w_1), \left|x\right|=\left|a\right|^{1-\theta}\left|b\right|^\theta\}\\
&= \inf \{ \|w_0a\|_X^{1-\theta}\|w_1b\|_X^\theta :  a\in X(w_0), b\in X(w_1),  \left|x\right|=\left|a\right|^{1-\theta}\left|b\right|^\theta\}.
\end{align*}

Standard lattice estimates such as \cite[Proposition 1.d.2]{lindtzaf-2} imply that
$$
\|x\|_\theta \geq \inf \{ \|w_0^{1-\theta}a^{1-\theta}w_1^\theta b^\theta\|_X :
\quad \left|x\right|=\left|a\right|^{1-\theta}\left|b\right|^\theta\} = \|x\|_{X(w_0^{1-\theta}w_1^\theta)},
$$
and the reverse inequality can be obtained by using $w_0a = w_1 b = w_0^{1-\theta}w_1^\theta x$.
\smallskip

To obtain $\Omega_\theta$ on $X_\theta$, we observe that a bounded homogeneous
selector for the evaluation map $\delta_\theta$ is defined by
$B_\theta(x) = (w_1/w_0)^{\theta -z} x$:
indeed, $B_\theta(x)(\theta) = x$ while $\|B_\theta x\| = \|x\|_{X_\theta}$ as it follows from
\smallskip

\quad\quad\quad $\|B_\theta(x)(0+it)\|_{w_0} = \|B_\theta(x)(1+it)\|_{w_1} =
\|w_0^{1-\theta}w_1^\theta x\|_X =\|x\|_{X_\theta}$.
\end{proof}

Complex interpolation between two Hilbert spaces always yields Hilbert spaces \cite{kalt-mon}.
Let us show that the induced derivation is trivial.

\begin{corollary}\label{trivialhilbert}
Let $(H_0, H_1)$ be an interpolation pair of Hilbert spaces.
Then for every $0<\theta<1$ the derivation $\Omega_\theta$ is trivial.
\end{corollary}
\begin{proof}
It follows from Proposition \ref{twist-weights}, since \cite[Lemma 2.2]{domast} shows that
$(H_0, H_1)$ is equivalent to an interpolation pair $(\ell_2(I), \ell_2(I, w))$, where $I$ is
a set and $w:I \to \R$ is a positive weight.
\end{proof}

Next we solve the stability problem for the splitting in the case of a pair of K\"othe spaces,
completing Theorem \ref{kalthm}.

\begin{theorem}\label{stsplit}
Let $(X_0, X_1)$ be an interpolation pair of superreflexive K\"othe function spaces and let $0<\theta<1$.
Then $\Omega_\theta$ is trivial if and only if there is a weight function $w$ so that $X_1 = X_0(w)$
up to equivalence of norms.
\end{theorem}


\begin{proof}
Recall that $X_0, X_1$ are spaces of $\mu$-measurable functions. The proof goes in two steps:

\noindent \textsc{Step 1.}
\emph{If $\Omega_\theta $ is trivial then there are weighted versions $Y_i$ of $X_i$ so that if $\Psi_\theta$ is the associated derivation, then there is a real function $f\in L_0(\mu)$ so that
$\Psi_\theta(x)- fx\in X_\theta$ and $\Psi_\theta- f$ is a bounded map on a dense subspace of $X_\theta$.}
\smallskip

Since we are dealing with interpolation of K\"othe function spaces, there is a positive function $k > 0$ such that
$\|x\|_{X_j} \leq \|kx\|_{\infty}$ for $j = 0, 1$. Consider the couple $(Y_0, Y_1)$, where $Y_j = X_j(1/k)$, $j = 0, 1$.
We denote the derivation induced at $\theta$ by this couple by $\Psi_{\theta}$. Then $Y_{\theta} = X_{\theta}(1/k)$
and $\Psi_{\theta}$ is trivial. Our advantage in working with $Y_{\theta}$ is that characteristic functions of measurable sets
are in this space.

Since $\Psi_\theta$ is a centralizer, there is a constant $c > 0$ such that for every
$a\in L_{\infty}(\mu)$ and every $x\in X$ we have $\|\Psi_\theta(ax) -a\Psi_\theta(x)\|_{X_{\theta}}\leq c\|a\|_{\infty}\|x\|_{Y_{\theta}}$,
and since it is trivial, there is a linear map $L$ so that $\Psi_\theta - L$ takes values in $Y_\theta$
and is bounded there.
The techniques in \cite{cfg} (Lemmas 3.10 and 3.13) show that after some averaging it is possible to get a
linear map $\Lambda$ such that $\Psi_\theta - \Lambda$ takes values in $Y_\theta$, is bounded there and
$\Lambda(ux)=u\Lambda x$ for every unit $u$ (every function with $|u|=1$).
Since characteristic functions can be written as the mean of two units one gets that if
$s = \sum_i \lambda_i 1_{A_i}$ is a simple function then $\Lambda(s x)=s\Lambda(x)$.
Now, simple functions are dense in $L_\infty$, so given $a\in L_\infty$ pick a simple $s$ so that
$\|a - s\|\leq \varepsilon$.
Since $\Lambda(ax) = \Lambda((a-s)x) +\Lambda (sx)$ and $a\Lambda(x) = (a-s)\Lambda(x)+ s\Lambda (x)$,
it follows that for some constant $K$
$$
\|\Lambda(ax)- a\Lambda(x)\|= \|\Lambda((a-s)x) -(a-s)\Lambda(x)\|\leq K\|a-s\|\|x\|\leq K\varepsilon\|x\|
$$
which shows that $\Lambda$ actually verifies $\Lambda(ax) = a\Lambda(x)$ for every $a\in L_\infty$.
It is then a standard fact that $\Lambda$ must have the form $\Lambda(x) = gx$ on the subspace $Y_{\theta}^b$ of bounded elements of $Y_{\theta}$.
\smallskip

Since $Y_{\theta}$ is superreflexive, it is $\sigma$-order continuous (Theorem 1.a.5 and Proposition 1.a.7 of \cite{lindtzaf-2}. So $Y_{\theta}^b$
is dense in $Y_{\theta}$.

Now, there is also $h > 0$ such that $\|hx\|_{L_1} \leq \|x\|_{Y_j}$, $j = 0, 1$. The centralizer $\Psi_{\theta}$ is bounded as a map from $Y_{\theta}$
into $Y_0 + Y_1$, so it is bounded from $Y_{\theta}$ into $L_1(h d\mu)$. The same is true of $\Psi_{\theta} - \Lambda$, so $\Lambda$ is bounded from
$Y_{\theta}$ into $L_1(h d\mu)$.

Since $Y_{\theta}^b$ is dense in $Y_{\theta}$, this means that $gh$ defines an element $x^*$ of $Y_{\theta}^*$. Since $Y_{\theta}$ is $\sigma$-order continuous, we have $Y_{\theta}^* = Y_{\theta}'$, and therefore there if $l$ such that $x^* = l$. It follows that
\[
\int_E gh d\mu = \int_E l d\mu
\]
for every measurable set $E$. This means that $l = gh$.

Let $x \in Y_{\theta}$, and take a sequence $(x_n) \subset Y_{\theta}^b$ such that $x_n \rightarrow x$. Then, by the considerations above,
taking limits in $L_1(hd\mu)$,
\[
\Lambda(x) = \lim_n \Lambda(x_n) = \lim_n gx_n = gx
\]

So $\Lambda(x) = gx$ for every $x \in Y_{\theta}$.

Write $g = g_1 + ig_2$, with $g_1, g_2$ real functions. Now, formula (\ref{kpgeneralized}) shows that the centralizer $\Psi_\theta$ is real. So, for
every $x \in Y_{\theta}$ real we have
\[
\|\Psi_{\theta}(x) - g_1 x\|_{Y_{\theta}} \leq \|\Psi_{\theta}(x) - gx\|_{Y_{\theta}} \leq C \|x\|_{Y_{\theta}}
\]
for some constant independent of $x$.

For $x \in Y_{\theta}$ write $x = x_1 + i x_2$, with $x_1, x_2$ real. Then, for some constant $C'$ independent of $x$
\begin{eqnarray*}
\|\Psi_{\theta}(x) - g_1 x\|_{Y_{\theta}} & \leq & \|\Psi_{\theta}(x) - \Psi_{\theta}(x_1) - \Psi_{\theta}(ix_2)\|_{Y_{\theta}} \\
			&& + \|\Psi_{\theta}(x_1) - g_1 x_1\|_{Y_{\theta}} + \|\Psi_{\theta}(x_2) - g_1 x_1\|_{Y_{\theta}} \\
			& \leq & C'(\|x_1\|_{Y_{\theta}} + \|x_2\|_{Y_{\theta}}) \\
			& \leq & 2C' \|x\|_{Y_{\theta}}
\end{eqnarray*}
where we have used the quasilinearity of $\Psi$ and the lattice properties of $Y_{\theta}$.

We take $f = g_1$.
\smallskip

\noindent \textsc{Step 2.} \emph{The spaces $Y_0, Y_1$ are weighted versions of each other.}
\smallskip

Pick  $w_0 = e^{\theta f}$ and $w_1= e^{(\theta - 1)f}$. By the previous proposition,
$$
(Y_\theta(w_0), Y_\theta(w_1))_\theta = Y_\theta(w_0^{1-\theta}w_1^\theta) = Y_\theta
$$
with associated centralizer $\Omega (x) = \log (w_0/w_1) x = f x = \Upsilon(x)$.
Thus $\Psi_\theta - \Omega$ is bounded and, by part (3) of Theorem \ref{kalthm}, we get
$Y_0 = X_\theta(w_0)$ and $Y_1= X_\theta(w_1)$, up to a renorming.
\end{proof}

Theorem \ref{stsplit} implies that the map $\Omega_\theta$, when trivial, is a bounded perturbation
of a multiplication map. This is a consequence of the symmetry properties of the K\"othe space.
Now we can complete Corollary \ref{trivialhilbert} with the following result stating that twisted
Hilbert spaces induced by interpolation of K\"othe spaces are trivial only in the obvious cases.

\begin{prop}\label{twistedhilbert}
A twisted Hilbert space induced by interpolation at $\theta=1/2$ between a superreflexive K\"othe space
$X$ and its dual is trivial if and only if for some weight function $w$ we have $X=L_2(w)$ with equivalence
of norms.
\end{prop}
\begin{proof}
If the twisted space is trivial then since $X_{1/2}=L_2$ (see, e.g., \cite{cfg}), and since spaces on the
whole scale are weighted versions of each other, $X$ and $X^*$ are equal to $L_2(w)$ and $L_2(w^{-1})$ with
equivalence of norms, respectively, for some weight.
\end{proof}

\subsection{Stability for families of three K\"othe spaces}\label{three}
%
Here we prove the global stability of both splitting and bounded splitting for interpolation families
consisting of three spaces distributed on arcs of $\T$.
The starting point is the generalization of the formula $X_\theta = X_0^{1-\theta}X_1^{\theta}$ for
families presented in \cite[Theorem 3.3]{kaltdiff} that Kalton credits to Hernandez \cite{Hernandez}.

In this section $\{A_1, ..., A_n\}$ will be a partition of $\T$ into arcs so that $A_i\cap A_j=\emptyset$
for $i\neq j$ and $\T= \cup_{j=1}^n A_j$.
Recall also that a K\"othe function space $X$ of $\mu$-measurable functions is said to be \emph{admissible}
\cite{kaltdiff} if $B_X$ is closed in $L_0(\mu)$ and there exist strictly positive $h,k\in L_0(\mu)$ such
that $\|xh\|_1\leq \|x\|_X\leq \|xk\|_\infty$ for each $x\in L_0(\mu)$.

\begin{defin}
Given K\"othe spaces $X_1, ..., X_n$ and positive numbers $a_1, ..., a_n$ we define
$$
\prod\limits_{j=1}^n X_j^{a_j} = \{f \in L_0 : \left |f\right| \leq
\prod\limits_{j=1}^n \left|f_j\right|^{a_j}, f_j \in X_j\}
$$
endowed with the norm $\|f\|_{\prod} = \inf\{\prod_{j=1}^n \|f_j\|_{X(j)}^{a_j}\}$, where the infimum
is taken over all choices of $f_j\in X_j$ so that $|f| \leq \prod_{j=1}^n \left|f_j\right|^{a_j}$.
\end{defin}

The following result provides the associated derivation map; we have included for the sake of clarity a
streamlined proof of the factorization theorem.

\begin{prop}\label{thm-centralizer}
Let $\{X_\omega\}_{\omega \in \T}$ be a strongly admissible family for which $X_\omega = X_j$ for
$\omega \in A_j$, $j = 1, ..., n$.
If $\mu_{z_0}$ denotes the harmonic measure on $\T$ with respect to $z_0$ one has
$$
X_{z_0} = \prod\limits_{j=1}^n X_j^{\mu_{z_0}(A_j)}.
$$
In particular, if $X$ is an admissible K\"othe function space, $w_j$ are weight functions
and $X_j = X(w_j)$, then the family $\{X_\omega\}_{\omega \in \T}$ as above is strongly admissible and
$X_{z_0} = X(\prod w_j^{\mu_{z_0}(A_j)})$ for $z_0\in \mathbb D$, with associated derivation
$\Omega_{z_0}(x) = - \Big( \sum\limits_j \psi_j'(z_0) \log w_j \Big) x$, where $\psi_j$ is an analytic
function on $\mathbb{D}$ such that $Re(\psi_j) = \chi_{A_j}$ on $\T$ and $\psi_j(z_0) = 0$.
\end{prop}
%
\begin{proof}
Pick $f \in X_{z_0}$. We are going to use \cite[Lemma 3.2 and Theorem 3.3]{kaltdiff}.
To this end recall that if $\mathcal E$ denotes the K\"othe function space on $\mathbb D\times \T$
with norm $\|\phi\|_{\mathcal E} = \esssup \|\phi(\cdot, e^{i\theta})\|_{X_{e^{i\theta}}}$
then there is $\phi\in \mathcal E$ so that $\|\phi\|_{\mathcal E} = \|f\|_{X_{z_0}}$ and
$$
\left|f(s)\right| = \exp \left(\int_{\T} P_{z_0}(\omega) \log \phi(s, \omega) d\omega\right).
$$

By Jensen's inequality
$$
\left|f(s)\right| \leq \prod_{j=1}^n \Bigg(\frac{1}{\mu_{z_0}(A_j)} \int_{A_j}
\phi(s, \omega) P_{z_0}(\omega) d\omega\Bigg)^{\mu_{z_0}(A_j)}.
$$

Set $f_j(s) = \frac{1}{\mu_{z_0}(A_j)} \int_{A_j} \phi(s, \omega) P_{z_0}(\omega) d\omega$ so that
$$
\|f_j\|_{X(j)} = \left\|\frac{1}{\mu_{z_0}(A_j)} \int_{A_j} \phi(\cdot,\omega) P_{z_0}(z) d\omega\right\|_{X(j)}
\leq  \frac{1}{\mu_{z_0}(A_j)} \int_{A_j} \|\phi(\cdot, \omega)\|_{X(j)} P_{z_0}(\omega) d\omega
\leq \|\phi\|_{\mathcal{E}}.
$$
Then $f_j \in X(j)$ and $\left|f\right| \leq \prod \left|f_j\right|^{\mu_{z_0}(A_j)}$, and thus
$\|f\|_{\prod} \leq\prod\|f_j\|_{X(j)}^{\mu(A_j)} \leq \|\phi\|$. So $\|f\|_{\prod} \leq \|f\|_{X_{z_0}}$.
Assume now that $\left|f\right| \leq \prod \left|f_j\right|^{\mu_{z_0}(A_j)}$, and let $\phi$ be given by
$\phi(s, \omega) = \prod \left|f_j(s)\right|^{\varphi_j(\omega)},$
where $\varphi_j$ is a harmonic function which coincides with $\chi_{A_j}$ on $\T$, $j = 1, ..., n$.
Then $\varphi_j(z_0) = \mu_{z_0}(A_j)$, and
\begin{eqnarray*}
\left|f(s)\right| & \leq & \prod \left|f_j(s)\right|^{\varphi_j(z_0)} =
\exp \left(\log \prod \left|f_j(s)\right|^{\varphi_j(z_0)}\right) \\
        & = & \exp\left( \sum \log \left|f_j(s)\right|^{\varphi_j(z_0)}\right)= 
        \exp \left(\sum \int_{\T} \log \left|f_j(s)\right|^{\varphi_j(\omega)} P_{z_0}(\omega) d\omega\right) \\
        & = & \exp \left(\sum \int_{A_j} \log \left|f_j(s)\right|^{\varphi_j(\omega)} P_{z_0}(\omega)d\omega\right) \\
        & = & \exp\left(\sum \int_{A_j} \log \prod_k \left|f_k(s)\right|^{\varphi_k(\omega)} P_{z_0}(\omega) d\omega\right) \\
        & = & \exp\left(\int_{\T} \log \prod \left|f_j(s)\right|^{\varphi_j(\omega)} P_{z_0}(\omega) d\omega\right)= 
        \exp\left(\int_{\T} \log \phi(s, \omega) P_{z_0}(\omega) d\omega\right).
\end{eqnarray*}

Therefore, $\|f\|_{X_{z_0}} \leq \|\phi\| = \max \|f_j\|_{X(j)}$.
If we multiply each $f_j$ by $\frac{\prod \|f_i\|_{X(i)}^{\mu_{z_0}(A_i)}}{\|f_i\|_{X(i)}}$ then
we still have that $\left|f(s)\right| \leq \prod \left|f_j(s)\right|^{\mu_{z_0}(A_j)}$ and
$\|f\|_{X_{z_0}} \leq \prod \|f_j\|_{X(j)}^{\mu_{z_0}(A_j)}$.
Since the functions $f_j$ are arbitrary, we get $\|f\|_{X_{z_0}} \leq \|f\|_{\prod}$.
\smallskip

For the second part, let $h_1$ and $k_1$ be functions that show that $X$ is admissible.
We set $h = h_1 \min w_j$ and $k = k_1 \max w_j$.
Then $h$ and $k$ are such that $\|xh\|_1\leq \|x\|_z\leq \|xk\|_{\infty}$ for every $x\in X$ and $z\in \T$.
Also $\|x w_j h_1\| \leq \|x\|_{X(w_j)} \leq \|x w_j k_1\|_{\infty}$.
Since it is clear that $B_{X(w_j)}$ is closed in $L_0$, each space $X(w_j)$ is admissible.

Selecting a countable dimensional dense subspace $Y$ of $X$, and taking as $V$ is the subspace generated
by $\{w_j^{-1}x : x \in Y, 1 \leq j \leq n\}$, the family $\{X_\omega\}_{\omega \in \T}$ is strongly
admissible with $V\cap B_{X_z}$ $L_0$-dense in $B_{X_z}$ for a.\ e.\ $z\in \T$.
Since $X_{z_0} = \prod X_j^{\mu_{z_0}(A_j)}$, for every $x \in X_{z_0}$ one has
\begin{eqnarray*}
\|x\|_{z_0} & = & \inf\left \{\prod \|x_j\|_{X_j}^{\mu_{z_0}(A_j)} :
     \left|x\right| \leq \prod \left|x_j\right|^{\mu_{z_0}(A_j)}\right\} \\
		& = & \inf\left\{\prod \|w_j x_j\|_X^{\mu_{z_0}(A_j)} : \left|x\right| \leq
   \prod \left|x_j\right|^{\mu_{z_0}(A_j)}\right\} \\
        & \geq & \inf\left\{\|\prod (w_j x_j)^{\mu_{z_0}(A_j)}\|_X : \left|x\right| \leq
   \prod \left|x_j\right|^{\mu_{z_0}(A_j)}\right\} \\
        & \geq & \inf \left\{\|\prod w_j^{\mu_{z_0}(A_j)} x\|_X : \left|x\right| \leq
   \prod \left|x_j\right|^{\mu_{z_0}(A_j)}\right\} \\
        & = &    \|x\|_{X(\prod w_j^{\mu_{z_0}(A_j)})},
\end{eqnarray*}
where the first inequality follows from $\|x^{\theta} y^{1 - \theta}\| \leq \|x\|^{\theta} \|y\|^{1-\theta}$
and an induction argument, and the second one from $\left|x\right| \leq \prod \left|x_j\right|^{\mu_{z_0}(A_j)}$.
Moreover, taking $f_j = \prod_k w_k^{\mu_{z_0}(A_k)} w_j^{-1} x$, we get the reverse inequality.
Now let
$$
F(z) = \prod_k w_k^{\mu_{z_0}(A_k)} \frac{x}{\prod_j w_j^{\psi_j(z)}}.
$$
Then $F \in \mathcal{N}^+(\mathcal{H})$, $F(z_0) = x$, and for $\omega \in A_j$ one has
$$
\|F(\omega)\|_{X_j} = \left\|w_j \prod w_k^{\mu_{z_0}(A_k)} \frac{x}{w_j}\right\|_X = \|x\|_{X_{z_0}}.
$$
Therefore $F$ is a $1$-extremal function for $x$. Moreover
$$
\Omega_{z_0}(x) =  F\;'(z_0)  =  - \prod w_j^{\mu_{z_0}(A_j)} x \prod\limits_k w_k^{-\psi_k(z_0)}
\sum\limits_j \psi_j'(z_0) \log w_j  =  - \Big( \sum\limits_j \psi_j'(z_0) \log w_j \Big) x,
$$
and the proof is done.
\end{proof}

We now pass to the study of the (bounded) stability.
We first observe that the extension to $\overline\D$ of a conformal transformion on $\D$ taking $z_0$
to $0$ takes an arc of $\T$ onto an arc of $\T$.
So we can assume without loss of generality that $z_0=0$, and so we will do throughout this section.
We will also consider three arcs $A_0 = [\theta_0, \theta_1)$, $A_1 = [\theta_1, \theta_2)$ and
$A_2 = [\theta_2, \theta_0)$ forming a partition of $\T$.

\begin{lemma}\label{lem-li}\label{lem-system}
For $j = 0,1,2$, set $\alpha_j = \frac{1}{2\pi}\int_{A_j} P_{0}(e^{it}) dt$ and
$\beta_j = \frac{1}{2\pi}\int_{A_j} e^{-it} dt$.
Then the vectors $a = (\alpha_0, \alpha_1, \alpha_2)$, $b = (Re(\beta_0), Re(\beta_1), Re(\beta_2))$
and $c = (Im(\beta_0), Im(\beta_1), Im(\beta_2))$
%
%
are linearly independent in $\R^3$. Consequently, we can find $a_j \in \mathbb{C}$ such that
$\sum a_j \alpha_j = 0$ and $\sum a_j \beta_j = -1$.
\end{lemma}
\begin{proof}
We begin by noticing that $\sum \alpha_j = 1$ and $\sum \beta_j = 0$.
So $a$ cannot be written as a linear combination of $b$ and $c$.
Also, the only way for $\{a, b, c\}$ to be linearly dependent is if $b$ is a multiple of $c$.
We have
$$
\frac{-i}{2}\beta_0 = - \sin \frac{\theta_1 - \theta_0}{2} \sin \frac{\theta_1 + \theta_0}{2} +
i \sin \frac{\theta_0 - \theta_1}{2} \cos \frac{\theta_0 + \theta_1}{2},\quad \textrm{and}
$$
$$
\frac{-i}{2}\beta_1 = - \sin \frac{\theta_2 - \theta_1}{2} \sin \frac{\theta_2 + \theta_1}{2}
+ i \sin \frac{\theta_1 - \theta_2}{2} \cos \frac{\theta_1 + \theta_2}{2}
$$
So, if we consider the matrix with lines $b$ and $c$, up to a factor of $\frac{-i}{2}$ the
determinant of the first two columns is
\begin{eqnarray*}
&-& \sin \frac{\theta_1 - \theta_0}{2} \sin \frac{\theta_1 + \theta_0}{2} \sin \frac{\theta_1 - \theta_2}{2}
   \cos \frac{\theta_1 + \theta_2}{2} \\
&+& \sin \frac{\theta_0 - \theta_1}{2} \cos \frac{\theta_0 + \theta_1}{2} \sin \frac{\theta_2 - \theta_1}{2}
   \sin \frac{\theta_2 + \theta_1}{2} \\
&=& \sin \frac{\theta_0 - \theta_1}{2} \sin \frac{\theta_2 - \theta_1}{2} \Bigg(\cos \frac{\theta_0 + \theta_1}{2}
   \sin \frac{\theta_2 + \theta_1}{2}- \sin \frac{\theta_1 + \theta_0}{2} \cos \frac{\theta_1 + \theta_2}{2} \Bigg) \\
&=& \sin \frac{\theta_0 - \theta_1}{2} \sin \frac{\theta_2 - \theta_1}{2} \sin \frac{\theta_2 - \theta_0}{2}
\end{eqnarray*}
which is zero if and only if two of the $\theta_j's$ are equal, which is not the case.
\end{proof}

The expressions for $\alpha_j$ and $\beta_j$ are provided by Theorems \ref{thm-interpolated-indicator}
and \ref{thm-indicator-of-centralizer}.

It follows from Kalton's Theorem \ref{kalthm} that given two interpolation couples $(X_0, X_1)$ and
$(Y_0, Y_1)$ such that $(X_0, X_1)_\theta = (Y_0, Y_1)_\theta$ (up to renorming) and
$\Omega_\theta = \Upsilon_\theta$ (up to a bounded map) for some $0<\theta<1$, then $X_0=Y_0$ and $X_1=Y_1$
(up to renorming).
Next we give the version for \emph{three spaces on arcs} of that result. The proof is essentially an adaptation of the proof of the uniqueness part of Theorem 7.6 of \cite{kaltdiff}. Compare with Theorem 7.9 of \cite{kaltdiff}, where the existence is provided.

\begin{theorem}\label{thm-unicity}
Let $\{X_\omega: \omega \in \mathbb T\} $ and $\{Y_\omega: \omega \in \mathbb T\}$ be two strongly
admissible families with $X_\omega = X^j$ and $Y_\omega = Y^j$ for $\omega \in \{e^{it} : t \in A_j\}$,
$j = 0, 1, 2$.
Let $\Omega_0$ and $\Psi_0$ be the corresponding derivations at $z_0=0$.
If $X_0 = Y_0$ (up to renorming) and $\Omega_0 - \Psi_0$ is bounded then $X^j = Y^j$  (up to renorming)
for $j=0,1,2$.
Moreover, $\Omega_z - \Psi_z$ is bounded for every $z\in \mathbb D$.
\end{theorem}
%
%
\begin{proof}
Since $\Omega_0$ and $\Psi_0$ are equivalent, so are $\Omega_0^{[1]}$ and $\Psi_0^{[1]}$ by definition,
and then
$$
d(\Phi^{\Omega_0}, \Phi^{\Psi_0}) =
\sup\limits_{\|f\| \leq 1, f \in \mathcal{I}} \left|\Phi^{\Omega_0}(f)-\Phi^{\Psi_0}(f) \right|< \infty.
$$

We can use now Theorems \ref{thm-interpolated-indicator}, \ref{thm-indicator-of-centralizer}
and Lemma \ref{lem-li} to get equations that determine $\Phi_{X^j}$ in terms of $\Phi_{X_0}$,
$Re(\Phi^{\Omega_0})$ and $Im(\Phi^{\Omega_0})$; and the same  for $\Phi_{Y^j}$ in terms of
$\Phi_{Y_0} = \Phi_{X_0}$, $Re(\Phi^{\Psi_0})$, and $Im(\Phi^{\Psi_0})$.
More specifically, on a suitable strict semi-ideal one has:
\begin{eqnarray*}
\alpha_0 \Phi_{X^0} + \alpha_1 \Phi_{X^1} + \alpha_2 \Phi_{X^2} & = & \Phi_{X_0} \\
Re(\beta_0) \Phi_{X^0} + Re(\beta_1) \Phi_{X^1} + Re(\beta_1) \Phi_{X^2} & = & Re(\Phi^{\Omega_0}) \\
Im(\beta_0) \Phi_{X^0} + Im(\beta_1) \Phi_{X^1} + Im(\beta_1) \Phi_{X^2} & = & Im(\Phi^{\Omega_0})
\end{eqnarray*}
\begin{eqnarray*}
\alpha_0 \Phi_{Y^0} + \alpha_1 \Phi_{Y^1} + \alpha_2 \Phi_{Y^2} & = & \Phi_{Y_0} \\
Re(\beta_0) \Phi_{Y^0} + Re(\beta_1) \Phi_{Y^1} + Re(\beta_1) \Phi_{Y^2} & = & Re(\Phi^{\Psi_0}) \\
Im(\beta_0) \Phi_{Y^0} + Im(\beta_1) \Phi_{Y^1} + Im(\beta_1) \Phi_{Y^2} & = & Im(\Phi^{\Psi_0})
\end{eqnarray*}

Lemma \ref{lem-li} establishes that there is a unique solution for the numerical system
\begin{eqnarray*}
\alpha_0 x + \alpha_1 y  + \alpha_2 z & = & a \\
Re(\beta_0) x + Re(\beta_1) y + Re(\beta_1) z & = & b \\
Im(\beta_0) x + Im(\beta_1) y + Im(\beta_1)z & = & c
\end{eqnarray*}
and two uniformly bounded sets of data $(a(x), b(x), c(x))$ and $(a'(x), b'(x), c'(x))$ with bounded
difference will produce two solutions $\Phi_{X^j}$ and $\Phi_{Y^j}$ with bounded difference.
So we can use \cite[Proposition 4.5]{kaltdiff}
%
to conclude that $X^j = Y^j$ (up to renorming) for $j=0,1,2$.
\end{proof}

Now, after a preparatory lemma, we consider the stability results for three K\"othe spaces.

\begin{lemma}\label{preparatory}
Let $X$ be a K\"othe function space and let $f\in L_0(\mu)$.
Then there are weight functions $\omega_j$ such that taking $Y_\omega = X(\omega_j)$ for
$\omega \in \{e^{it} : t \in A_j\}$ and $j = 0, 1, 2$, the strongly admissible family
$\{Y_\omega : \omega \in \mathbb{T}\}$ yields $Y_0 = X$ with (linear) derivation
map $\Omega_0 (x)= f x$.
\end{lemma}
\begin{proof}
Write $f = f_1 + i f_2$. From Lemma \ref{lem-system} we have that there are real numbers
$a_0, a_1, a_2$ such that $\sum a_j \alpha_j = 0$ and $\sum a_j \beta_j = -1$.
Also, there are real numbers $b_0, b_1, b_2$ such that $\sum b_j \alpha_j =0$ and $\sum b_j \beta_j =-i$.
Set $w_j = e^{a_j f_1 + b_j f_2}$. Then $Y_0 = X(w_0^{\alpha_0} w_1^{\alpha_1} w_2^{\alpha_2}) = X$
and $\Omega_{0} = \sum -\beta_j (a_1 f_1 + b_j f_2) = f$.
\end{proof}

\begin{theorem}\label{weighthree}
Let $\{X_\omega : \omega \in \mathbb{T}\}$ be a (strongly) admissible family with $X_\omega = X^j$ for
$\omega \in \{e^{it} : t \in A_j\}$ and $j=0,1,2$.
If the derivation map $\Omega_0$ is trivial then there are weight functions $w_j$ such that
$X^j = X_0(w_j)$ with equivalence of norms.
In particular, $\Omega_z$ is trivial for every $z \in \mathbb{D}$.
\end{theorem}
\begin{proof}
If $\Omega_0$ is trivial then we argument as in Theorem \ref{stsplit} to get (up to passing to weighted versions of the spaces) some $f\in L_0(\mu)$
so that the linear map $\Lambda(x)= fx$ on $X_0$ satisfies that $\Omega_0 - \Lambda$ takes $X_0$
into $X_0$ and it is bounded there.

We take $X = X_0$ in Lemma \ref{preparatory} so that we obtain a a new family
$\{Y_\omega : \omega \in \mathbb{T}\}$ which has $\Lambda$ as induced centralizer at $0$.
By Theorem \ref{thm-unicity}, we obtain that $X^j = X_0(w_j)$ with equivalence of norms for some
suitable weights.

To see that $\Omega_z$ is trivial for every $z \in \mathbb{D}$, observe that the Kalton spaces
associated to the families $\{X_\omega\}$ and $\{Y_\omega\}$ coincide, with equivalence of norms.
We will denote these Kalton spaces $\mathcal N^+(\{X_\omega\})$ and $\mathcal N^+(\{Y_\omega\})$.
In particular, $X_z = Y_z$ for all $z\in \mathbb D$, with equivalence of norms.

Let $\Lambda_z$ be the trivial centralizer induced by $\{Y_\omega\}$ at $z$.
Given $x \in V$, we fix $(1+\epsilon)$-extremals $F_x \in \mathcal{N}^+(\{X_\omega\})$ and
$G_x \in \mathcal{N}^+(\{Y_\omega\})$  so that $\Omega_z(x)=\delta_{z}'F_x$ and
$\Lambda_z(x)= \delta_{z}'G_x$.
Thus there are some constants $C, C'$ such that, for all $x \in V$, one has
\begin{eqnarray*}
\|\Omega_z(x) - \Lambda_z(x)\|_{X_z} & = & \|\delta_{z}'(F_x - G_x)\|_{X_z} \\
	& \leq & \|\delta_{z}': \ker \delta_{z} \rightarrow X_z\| \|F_x - G_x\|_{\mathcal{N^+} }\\
    & \leq & C (\|F_x\|_{\mathcal{N^+}(\{X_\omega\})} + \|G_x\|_{\mathcal{N^+}(\{Y_\omega\})}) \\
    & \leq & C (1+\epsilon) (\|x\|_{X_z} + \|x\|_{Y_z}) \\
    & \leq & C' \|x\|_{X_z}.
\end{eqnarray*}

Since $\Lambda_z$ is trivial and $V$ is dense in $X_z$, the derivation $\Omega_z$ is trivial.
\end{proof}

We should note that our reasoning does not work if we take general families of three spaces, i.\ e.\, not necessarily distributed in arcs. This should be compared with the results of \cite{yanqui}, previously obtained in \cite{cj}.

\subsection{(Bounded) stability fails for families of four K\"othe spaces}\label{four}


Here we show that the statement of Theorem \ref{kalthm} is no longer true for arbitrary families of K\"othe spaces.

A sequence $(\varphi_n)$ of functions which are continuous on $\overline{\mathbb{D}}$ and
analytic on $\mathbb{D}$ induces a family of diagonal linear maps $D(z): c_{00} \rightarrow c_{00}$
($z\in \overline{\mathbb{D}}$) given by $D(z)(x_n)= \left(\varphi_n(z)x_n\right)$.
We define a family of Banach spaces $\{X_s: s\in \T\}$ by taking $X_s$ the
completion of $c_{00}$ with respect to the norm $ \|x\|_s = \|e^{-D(s)}x\|_2$.
Moreover, for $x \in c_{00}$, we define $\|x\|_{\Sigma} = \inf\{\|x_1\|_{z_1}+\cdots+\|x_n\|_{z_n}\}$,
where the infimum is taken over all $n \in \mathbb{N}$, $z_i \in \T$, and
$x_i \in c_{00}$ such that $x = x_1 +\cdots+ x_n$.
\smallskip

\emph{We claim that $\|.\|_{\Sigma}$ is a norm on $c_{00}$.}
Indeed, the only difficulty is to show that $\|x\|_{\Sigma} = 0$ implies $x = 0$.
Let $x = (a^j) \in c_{00}$ with $a^k \neq 0$.
Note that $e^{-D(z)}$ is the multiplication operator associated to the sequence $(e^{-\varphi_n(z)})$.
If $\left|\varphi_k(z)\right| \leq M$,
then $\left|e^{-\varphi_k(z)}\right| = e^{-Re(\varphi_k(z))} \geq e^{-M}$.
Therefore
$$
\sum \|x_j\|_{z_j} = \sum \|e^{-D(z_j)}x_j\|_2 \geq
e^{-M}\left|a^k\right|,
$$
and we conclude that $\|x\|_{\Sigma} > 0$.
\smallskip

Let $\Sigma$ be the completion of $c_{00}$ with respect to $\|.\|_{\Sigma}$.
Then for each $\omega \in \T$ we have $X_\omega\subset \Sigma$ with inclusion having norm at most $1$.
Note also that the projection $P_n$ onto the first $n$ coordinates is a norm-one operator on
$X_\omega$ for each $\omega \in \T$, and also on $\Sigma$.

\begin{prop}\label{omegaoper}
The above defined family $(X_\omega)_{\omega\in \T}$ is an interpolation family with containing
space $\Sigma$ and intersection space $\Delta=c_{00}$.
Moreover, for every $z_0$ in $\mathbb{D}$ one has:
\begin{enumerate}
\item $X_{\{z_0\}} = X_{[z_0]}$. Thus we can denote $X_{z_0} = X_{\{z_0\}} = X_{[z_0]}$.
\item The space $X_{z_0}$ is the completion of $c_{00}$ with respect to the norm
$\|x\|_{z_0} = \|e^{-D(z_0)}x\|_2$.
\item $\Omega_{z_0}x = D'(z_0) x$ for $x\in c_{00}$.
\end{enumerate}
\end{prop}
\begin{proof}
(1)  Let $x \in c_{00}$. Clearly $\|x\|_{[z_0]} \leq \|x\|_{\{z_0\}}$.
Let $f \in \mathcal{F}$ (see Section \ref{sect:admissible}) be such that $f(z_0) = x$.
Take $n$ such that $P_n(x)=x$ and define $g(z) = P_n(f(z))$.
Then $g(z) = \sum_{j=1}^n \psi_j(z) e_j$, where $(e_j)$ is the canonical basis of $\ell_2$.
Since $\psi_j(z) e_j = (P_j - P_{j-1}) f(z)$ and $f$ is analytic when viewed as a $\Sigma$-valued
function, we get that $\psi_j$ is analytic.
If $z \in \mathbb{D}$, then
$$
\left|\psi_j(z)\right| \|e_j\|_{\Sigma} = \|(P_j - P_{j-1})f(z)\|_{\Sigma} \leq 2\|f(z)\|_{\Sigma}
\leq 2\|f\|_\mathcal{F}.
$$
Hence $\psi_j \in H^{\infty}$, the space of bounded analytic functions on $\mathbb{D}$, which is
contained in the Smirnov class $N^+$.
Also, for almost every $z \in \T$ we have
$\|g(z)\|_{X_z} = \|P_n(f(z))\|_{X_z} \leq \|f(z)\|_{X_z} \leq \|f\|_\mathcal{F}$.
Thus $g \in \mathcal{G}$ and $\|g\|_\mathcal{G} \leq \|f\|_\mathcal{G}$.
Since $g(z_0) = P_n(f(z_0)) = x$, we get $\|x\|_{\{z_0\}} = \|x\|_{[z_0]}$.

To prove (2), let $x \in c_{00}$ and let $g(z) = e^{D(z) - D(z_0)}x \in \mathcal{G}$.
Then $g(z_0) = x$, and for $z \in \mathbb{T}$ we have $\|g(z)\|_{X_z} = \|e^{-D(z_0)}x\|_2$.
Thus $\|x\|_{z_0} \leq \|e^{-D(z_0)}x\|_2$.
Take $f \in \mathcal{G}$ such that $f(z_0) = x$.
Given a non-zero $y \in c_{00}$, define
$$
h(z) = \langle e^{-D(z)}f(z), y\rangle.
$$
It follows that $h \in H^{\infty}$.
Indeed, $f$ may be written as a finite sum $\sum f_i x_i$, with $f_i \in N^+$ and $x_i \in c_{00}$.
Since $e^{-D(z)}$ is bounded, we have that $e^{-D(z)} f_i(z) \in N^+$.
This implies that $h \in N^+$, and since it is bounded on $\mathbb{T}$, $h \in H^{\infty}$
\cite[Theorem 2.11]{duren}. Moreover, $\|h\|_{H^{\infty}} \leq \|f\|_\mathcal{G}\|y\|_2$.
Hence $\left|h(z_0)\right| = \left|\langle e^{-D(z_0)}x, y\rangle \right| \leq \|f\|_\mathcal{G}\|y\|_2$.
Since $\|f\|_\mathcal{G}$ can be taken arbitrarily close to $\|x\|_{z_0}$ and $y$ is arbitrary,
$\|e^{-D(z_0)}x\|_2 \leq \|x\|_{z_0}$.

(3) We have shown that the following function $g$ is an extremal function for $x=(a^j)$ at $z_0$:
$$
g(z) = e^{D(z)-D(z_0)}x = (e^{\varphi_1(z) -\varphi_1(z_0)}a^1, e^{\varphi_2(z) -\varphi_2(z_0)}a^2,\ldots).
$$
Then $g'(z) = \left(\varphi'_n(z) e^{\varphi_n(z) - \varphi_n(z_0)}a^n\right)$, hence
$\Omega_{z_0}x = g'(z_0) = \left(\varphi'_n(z_0)a^n\right) = D'(z_0)x.$
\end{proof}

By Proposition \ref{omegaoper}, there is no local bounded stability for arbitrary
families of K\"othe spaces:

\begin{prop}\label{cbounded}
Let  $D(x_n) = (w_n x_n)$ be an unbounded diagonal operator on $c_{00}$.
\begin{enumerate}
\item The choice $D(z) = zD$ yields an analytic family for which $\Omega_z = D$ for every $z\in\mathbb{D}$.
\item The choice $D(z) =z^2D$ yields an analytic family such that $\Omega_z = 2zD$ for every $z\in\mathbb{D}$.
Therefore $\Omega_0 = 0$ while $\Omega_z$ is unbounded for $z \neq 0$.
\end{enumerate}
\end{prop}




We pass now to show that there is no local stability for families of K\"othe spaces.

\begin{prop}
Let $p: \T \to [1,\infty)$ be a measurable function, let $\alpha: \overline{\mathbb D} \to \C$ be a
function analytic on $\mathbb{D}$ and satisfying $Re \left(\alpha(z)^{-1}\right) = p(z)^{-1}$ on $\T$.
We consider the interpolation family $(\ell_{p(\omega)})_{\omega \in \T}$.
Given $z_0\in\mathbb D$ with $\alpha(z_0) \in \mathbb{R}$, the interpolation space at $z_0$ is
$\ell_{p(z_0)}$ with derivation
$$
\Omega_{z_0}\left((x_n)\right) = -\frac{\alpha'(z_0)}{\alpha(z_0)}
\left(x_n \log \frac{\left|x_n\right|}{\|x\|_{\ell_{p(z_0)}}}\right).
$$
\end{prop}
\begin{proof}
The containing space for the family is $\ell_{\infty}$, and the intersection space may be taken as $c_{00}$.
Notice that even if we take the biggest intersection space possible, we still would have that $c_{00}$ is
dense in the interpolation space. We first check that
$$
f(z) = \left(\left|x_n\right|^{\frac{\alpha(z_0)}{\alpha(z)}} \frac{x_n}{\left|x_n\right|}\right)
$$
is a $1$-extremal for $x=(x_n) \in c_{00}$ with $\|x\|_{p(z_0)}=1$.
The function $f$ is analytic, $f(z_0) = x$, and $f \in \mathcal{G}$ because for every $z \in \mathbb{D}$
$$
\|f(z)\|_{p(z)}^{p(z)}  =  \sum \left|x_n\right|^{Re(\frac{\alpha(z_0)}{\alpha(z)}) p(z)}
 = \sum \left|x_n\right|^{\alpha(z_0)} = \sum \left|x_n\right|^{p(z_0)} = 1.
$$
Therefore, $\|x\|_{z_0} \leq \|x\|_{p(z_0)}$.
The reverse inclusion is proved by a standard argument (see \cite[5.5.1. Theorem]{BL}).
Moreover for non-zero $x=(x_n)\in \ell_{p(z_0)}$ one has
\begin{eqnarray*}
\Omega_{z_0}(x) & = & \|x\|_{\ell_{p(z_0)}}\Omega_{z_0}\left(\frac{x}{\|x\|_{\ell_{p(z_0)}}}\right) \\
& = & \|x\|_{\ell_{p(z_0)}} \left(- \frac{x_n}{\left|x_n\right|}
\left(\frac{\left|x_n\right|}{\|x\|_{\ell_{p(z_0)}}}\right)^{\frac{\alpha(z_0)}{\alpha(z_0)}}
\frac{\alpha(z_0)}{\alpha(z_0)^2} \alpha'(z_0) \log\frac{\left|x_n\right|}{\|x\|_{\ell_{p(z_0)}}}\right)\\
& = & - \frac{\alpha'(z_0)}{\alpha(z_0)} \left(x_n \log \frac{\left|x_n\right|}{\|x\|_{\ell_{p(z_0)}}}\right),
\end{eqnarray*}
and the proof is complete.
\end{proof}

An exact sequence is \emph{singular} when the quotient map $q$ is strictly singular; i.\ e.\, no restriction
of $q$ to an infinite dimensional subspace is an isomorphism.
A derivation is said to be \emph{singular} if the induced exact sequence is singular \cite{ccs}.
Obviously, singular derivations are not trivial.

\begin{prop}\label{singfam}
The family $(\ell_{p(z)})_{z \in \mathbb{T}}$  with $p(z)^{-1} = Re\left((z^2 + 2)^{-1}\right)$
yields $\Omega_0 = 0$ and $\Omega_z$ singular for $0\neq z \in \mathbb{D}$.
\end{prop}
\begin{proof}
Since $p(z)^{-1} \in [1/3, 1]$
it turns out that $p(z) \in [1, 3]$. We thus set $\alpha(z) = z^2 + 2$ on $\mathbb{D}$.
In that case we get $\alpha(z) \in \mathbb{R}$ if and only if $z = t$ or $z= it$, $t \in \mathbb{R}$.
By the previous lemma,  $\Omega_0 = 0$, and for $z = t$ and $z = it$, $t \neq 0$, $\Omega_{z}$ is a
nonzero multiple of the Kalton-Peck map on $\ell_{p(z_0)}$, and therefore it is singular.
Moreover, the choice $\alpha_{z_0}(z) = z^2 + 2 - i \mathrm{Im}(\alpha(z_0))$ yields that $\Omega_{z_0}$
is a nonzero multiple of the Kalton-Peck map for any $z_0 \in \mathbb{D}$, $z_0 \neq 0$.
\end{proof}

\noindent\textbf{The moral of all this.}
We can present two explanations for the fact that families of two or three K\"othe spaces have global
(bounded) stability and are even rigid in different senses while families of four spaces do not.
The first one emerges from the proof of Theorem \ref{thm-unicity}: any point in the interior of the
convex hull of two or three points admits a unique representation as a convex combination of them,
which is false for four points.
The second one arises from the reiteration theorem for families \cite{Coifman1982}.
Using that result to set the initial configuration one gets:

\begin{theorem} \label{interpcoupfam2}
Let $\alpha$ and $(X_0, X_1)_{\alpha(\omega)}$ for $\omega\in \T$ be as in Theorem \ref{interpcoupfam},
and let $\Omega_{s}$ denote the derivation corresponding to $(X_0, X_1)_s$ for $0<s<1$.
Then the derivation corresponding to the family $(X_0, X_1)_{\alpha(z)}$ at $z\in \mathbb{D}$ is
$\Phi_{z} = w'(z) \Omega_{\alpha(z)}$, where $w = \alpha + i\tilde{\alpha}$ and $\tilde{\alpha}$ is
the harmonic conjugate of $\alpha$ with $\tilde{\alpha}(z) = 0$.
\end{theorem}
\begin{proof}
Fix $z\in \mathbb D$ and $x \in X_0 \cap X_1$, and take a $c$-extremal $f$ for $x$ at $\alpha(z)$ in
the Calderon space $\mathscr{C}(X_0, X_1)$.
By \cite[4.2.3. Lemma]{BL} we may assume that $f$ is a linear combination of functions
with values in $X_0 \cap X_1$.
Included in the proof of \cite[Theorem 5.1]{Coifman1982} is the fact that $g = f \circ w$ is an
extremal for $x$ at $z$ with respect to the family $(X_0, X_1)_{\alpha(z)}$, and $\|g\| \leq \|f\|$.
Therefore $\Phi_{z}(x) = (f\circ w)'(z) = w'(z) \Omega_{w(z)}(x)$.
Finally $\Omega_{w(z)}$ may be chosen as $\Omega_{\alpha(z)}$ by vertical symmetry.
\end{proof}

This result can be seen as the $1$-level version of the reiteration Theorem \ref{interpcoupfam}.
It shows that, under the hypothesis of Theorem \ref{interpcoupfam}, the derivation map of the family
is always a multiple of the derivation map of the initial pair.


\begin{corollary}\label{familyofKP}
Let $(X_0, X_1)$ be an interpolation pair, let $\Omega_\theta$ be the derivation at $\theta \in (0,1)$,
let $B_0 = \{e^{i\theta} :\theta\in [0,\frac{\pi}{2}] \cup [\pi,\frac{3\pi}{2}]\}$ and let
$\alpha =\chi_{B_0}:\T\to [0,1]$.
Consider the interpolation family $\{(X_0, X_1)_{\alpha(w)} : w \in \T\}$.
Then for $z = t$ or $z = it$, $t \in (-1, 1)$, we get $X_z = (X_0, X_1)_{\frac{1}{2}}$ with
derivations $\Phi_0=0$ and $\Phi_z$ equal to a multiple of $\Omega_\theta$ with $\theta=\alpha(z)$
for $0\neq z \in \mathbb{D}$.
\end{corollary}

A case similar to Proposition \ref{singfam} can be obtained with just two spaces distributed on
four arcs on $\T$ as above: just consider $X_0 = \ell_{\infty}$ and $X_1 = \ell_1$, which
produces $X_z = \ell_2$ for every $z = t$, $z = it$, $t \in (-1, 1)$ and $\Phi_0 = 0$ while
$\Phi_z$ is a non-zero multiple of the Kalton-Peck map on $X_z$ for values of $z \in \D$
arbitrarily close to $0$. Thus, the differential process lacks local stability.
\smallskip

The next result explains, to some extent, the exceptional character of the previous examples.
It is a direct consequence of Theorem \ref{interpcoupfam2}.

\begin{theorem}\label{stabilityfam}
Let $(X_0, X_1)$ be an interpolation pair of K\"othe function spaces, and let us consider
the notation of Theorem \ref{interpcoupfam2}.
\begin{enumerate}
\item If the derivation $\Phi_{z_0}$ is bounded for some $z_0\in\mathbb D$ such that $0<\alpha(z_0)<1$
and $w'(z_0)\neq 0$, then $X_0 = X_1$ with equivalence of norms and $\Phi_z$ is bounded for each
$z\in \mathbb D$.
\item If the derivation $\Phi_{z_0}$ is trivial for some $z_0\in\mathbb D$ such that $0<\alpha(z_0)<1$
and $w'(z_0)\neq 0$, then $X_0$ is a weighted version of $X_1$ and $\Phi_z$ is trivial for each
$z\in \mathbb D$.
\end{enumerate}
\end{theorem}

\section{Stability of splitting for general Banach spaces}\label{stabilitygen}

For general Banach spaces, the problem of existence of local or global stability remains open.
Here we give some positive results for pairs of sequence spaces with a common basis and pairs of K\"othe
function spaces.
In the latter case, they provide more information than the results given before, which have an isomorphic
nature, while Theorem \ref{unificado}, under the additional hypotheses it imposes, yields isometric
uniqueness and stability.

In this section, given an interpolation pair $\overline X =(X_0,X_1)$, we take
$\mathscr{F}^\infty(\overline{X})\equiv\mathscr{F}^\infty(\mathbb{S},\Sigma)$ as Kalton space (see
Section \ref{sect:admissible}), and the $c$-extremals $f_{x,\theta}$ for $x$ at $\theta$ belong to
$\mathscr{F}^\infty(\overline{X})$.
Sometimes we denote $B_\theta(x)(z)= f_{x,\theta}(z)$ ($z\in \mathbb{S})$ for notational convenience.

%
%

Recall that, given $0<\theta<1$ and $t\in \mathbb{R}$, the invariance under vertical translations of
$\mathbb{S}$ implies that given $f$ in the Calderon space $\mathscr{C}$ such that $f(\theta)=x$, the
function $g(z)=f(z-it)$ is in $\mathscr{C}$ and satisfies $\|f\|_\mathscr{C}=\|g\|_\mathscr{C}$ and
$g(\theta +it) = x$; and the same is true for $\mathscr{F}^\infty(\overline{X})$.
Thus $X_\theta =X_{\theta+it}$ isometrically, and it is enough to study the scale
$\left(X_\theta\right)_{0<\theta<1}$.

Our analysis begins with an observation about the properties of the map $\theta \to \|\cdot\|_\theta$.
Recall that an interpolation pair $(X_0,X_1)$ is \emph{regular} if $\Delta$ is dense in both $X_0$
and $X_1$.

\begin{lemma}\label{continuity}
Let $(X_0,X_1)$ be a regular interpolation pair and let $0\leq\theta_0<\theta_1\leq 1$.
For every $x\in X_{\theta_0}\cap X_{\theta_1}$, the map $\theta \mapsto \|x\|_\theta\in\mathbb{R}$
is log-convex on $(\theta_0,\theta_1)$; it is therefore continuous with right and left derivatives on
any point of $(\theta_0,\theta_1)$.
\end{lemma}
\begin{proof}
For each $\theta \in [\theta_0,\theta_1]$ one has
$\|x\|_\theta \leq \|x\|_{\theta_0}^{1-t} \|x\|_{\theta_1}^t$ when $\theta=(1-t)\theta_0+t\theta_1$:
the case $\theta_0=0$, $\theta_1=1$ is well-known, and the general case is a consequence of the
reiteration theorem for  complex interpolation \cite[4.6.1. Theorem]{BL}.
From this it follows that the map $\theta \mapsto \log \|x\|_\theta$ is convex on $[\theta_0,\theta_1]$,
and therefore continuous with right and left derivatives at every point of $(\theta_0,\theta_1)$.
\end{proof}

\subsection{Local bounded stability for coherent pairs} \label{sect:Lafforgue}
The coherent pairs in the title are those satisfying the thesis of Proposition \ref{ri-space}.

The proof of the following result is a part of the proof of \cite[Th\'eor\`eme]{daher}.
We include some details for the convenience of the reader.

\begin{lemma}\label{lemma:extremal}
Given $(X_0, X_1)$ a regular interpolation pair with $X_0$ reflexive, $x\in\Delta$, $\theta\in (0,1)$ and
a $1$-extremal $f_{x,\theta}$ one has $\|f_{x, \theta}(z)\|_z = \|x\|_\theta$ for every $z \in\mathbb{S}$.
\end{lemma}
\begin{proof}
It is enough to prove the result when $\|x\|_\theta=\|f_{x,\theta}\|_{\mathscr{F}^\infty(\overline{X})}=1$.
We select $x^*\in (X_\theta)^*\equiv (X^*)_\theta$ such that $\|x^*\|=\langle x,x^*\rangle=1$.
As in Daher's proof, we select $f^*\in \mathscr{F}_\theta^2(\overline{X^*})$ with $f^*(\theta)=x^*$ and
$\|f^*\|_{\mathscr{F}_\theta^2(\overline{X^*})}=1$.

Using \cite[4.2.3. Lemma]{BL} we can show that $g(z)=\langle f_{x, \theta}(z),f^*(z)\rangle$ defines
an analytic function.
Since $|g(z)|\leq 1$ for every $z \in \mathbb{S}$ and $g(\theta)=1$, the maximum principle for analytic
functions implies that $g(z)= 1$ for every $z \in \mathbb{S}$.
In particular $\|f_{x, \theta}(z)\|_z \geq 1$, hence $\|f_{x, \theta}(z)\|_z = 1$.
\end{proof}


\begin{lemma}\label{locallemma}
Let $(X_0,X_1)$ be a regular interpolation pair with $X_0$ reflexive, let $x \in \Delta$ and let
$0\leq \theta_0 <\theta< \theta_1 \leq 1$.
Suppose that there is a $1$-extremal $f_{x, \theta}$ which is derivable at $z=\theta$ as a function
with values in both spaces $X_{\theta_i}$ ($i=0,1$), and consider the derivation
$\Omega_\theta(x) = f_{x, \theta}'(\theta)$.
Then the right and left derivatives of $t \mapsto \|x\|_t$ at $\theta$ are bounded in modulus by
$\left \|\Omega_{\theta} (x) \right\|_{\theta}$.
\end{lemma}
\begin{proof}
By Lemma \ref{lemma:extremal} $\|x\|_\theta = \|f_{x,\theta}(\theta+\varepsilon)\|_{\theta + \varepsilon}$.
Hence
\begin{eqnarray*}
 \lim_{\varepsilon\to 0^+} \frac{1}{\varepsilon} \big| \|x\|_{\theta + \varepsilon} - \|x\|_{\theta}\big|
 &\leq & \limsup_{\varepsilon\to 0^{+}} \frac{1}{\varepsilon}
 \left(\|f_{x,\theta}(\theta+\varepsilon) -x \|_{\theta+\varepsilon}\right)\\
&\leq & \limsup_{\varepsilon\to 0^{+}}  \left( \|\Omega_\theta(x)\|_{\theta+\varepsilon}+
\|\frac{1}{\varepsilon} (f_{x,\theta}(\theta+\varepsilon)-x)-\Omega_\theta(x) \|_{\theta+\varepsilon}\right)\\
&\leq & \limsup_{\varepsilon\to 0^{+}}  \left( \|\Omega_\theta(x)\|_{\theta+\varepsilon}+
\max_{i=0,1} \|\frac{1}{\varepsilon} (f_{x,\theta}(\theta+\varepsilon)-x)-\Omega_\theta(x) \|_{\theta_i}\right).
\end{eqnarray*}
Note that $\Omega_\theta(x)$ belongs to $X_{\theta_0} \cap X_{\theta_1}$ by hypothesis.
So, by Lemma \ref{continuity}, we have that $\|\Omega_\theta(x)\|_{\theta+\varepsilon}$ tends to
$\|\Omega_\theta(x)\|_{\theta}$.
Since $\frac{1}{\varepsilon} (f_{x,\theta}(\theta+\varepsilon) -x)$ tends to $\Omega_\theta(x)$
in $X_{\theta_i}$, $i=0,1$, we get:
\begin{equation}\label{difestimate}
\left|\frac{d\|x\|_t}{dt}|_{t=\theta^{\pm}}\right| \leq  \left \| \Omega_{\theta} (x) \right\|_{\theta}.
\end{equation}
\end{proof}

The estimate (\ref{difestimate}) points out to the fact that the scale $(\Omega_t)_{0<t<1}$ seems to act
as ``the infinitesimal generator of the group of natural uniform homeomorphisms $B_t(.)(s) : X_t \to X_s$
between the spheres of the interpolation spaces", as it appears in \cite{daher}.

Next we give some conditions on an interpolation pair $(X_0, X_1)$ implying $X_0=X_1$ up to an
equivalent renorming.

\begin{prop}\label{mono}\label{ri-space}
Let $(X_0, X_1)$ be a regular interpolation pair of reflexive spaces.
Suppose that
\begin{enumerate}
\item $X_0$ and $X_1$ have a common monotone basis $(e_n)$, or
\item $X_0$ and $X_1$ are rearrangement invariant spaces on $[0,1]$.
\end{enumerate}
Then there is an increasing sequence $(E_n)$ of finite dimensional subspaces of $\Delta$ with
$\Delta_0=\cup_{n\in\N} E_n$ dense in $\Delta$, such that for every $x\in E_n$ we can select a
$1$-extremal $f_{x,\theta}$ so that the corresponding derivation map $\Omega_\theta$ takes $E_n$
into $E_n$ and is bounded on $E_n$.
%
\end{prop}
\begin{proof}
Given $0<\theta<1$ and $x \in X_\theta$, there exists a $1$-extremal $g_{x,\theta}$ by
\cite[Proposition 3]{daher}.

(1) Take $E_n=[e_1,\ldots,e_n]$ and denote by $P_n$ the natural norm-one projection
from $\Sigma$ onto $E_n$.
For $x \in E_n$, if $g_{x, \theta}$ is a $1$-extremal then
$f_{x, \theta}(z) = P_n \left( f_{x,\theta}(z)\right)$ defines a $1$-extremal that satisfies
the remaining conditions because all norms are equivalent on $E_n$ and for $y\in E_n$
$$
\Omega_\theta(y)= g_{y,\theta}'(\theta) =  \left( P_n f_{y,\theta} \right)'(\theta) =
P_n \left( f_{y,\theta}'(\theta)\right).
$$

(2) The proof is similar: For each $n\in \N$ we take as $E_n$ the subspace generated by the
characteristic functions of the intervals $\big((k-1)/2^n,k/2^n\big)$, $k=1,\ldots,2^n$.
The arguments in the proof of \cite[Theorem 2.a.4]{lindtzaf-2} show that
$$
P_n f =\sum_{k=1}^{2^n} 2^n \left(\int_0^1 f \chi_{n,k} dt\right) \chi_{n,k}
$$
define a norm-one projection onto $E_n$.
\end{proof}

\begin{theorem}[Local bounded stability]\label{local}
Let $(X_0,X_1)$ be an interpolation pair of spaces as in Proposition \ref{ri-space}
and let $0\leq \theta_0<\theta_1\leq 1$.
Suppose that $\sup_{\theta_0 < t < \theta_1}\|\Omega_t: X_t \to X_t\|<\infty$.
Then $X_0 = X_1$, up to an equivalent renorming.
\end{theorem}
\begin{proof}
Fix $x\in\Delta_0$. For $\theta_0 < s < \theta_1$ one has
$$
\left |\frac{d\|x\|_t}{dt}\right|_{t=s^{+}} \leq  \left\|\Omega_{s} (x) \right\|_s \leq M\|x\|_s.
$$

If we set $g(s)=e^{Ms}\|x\|_s$ then
$$
\left(\frac{dg}{dt}\right)_{t=s^+}= e^{Ms}\left(M\|x\|_s+\left(\frac{d\|x\|_t}{dt}\right)_{t=s^{+}}\right)
\geq 0.
$$
Since $g$ is continuous, it is nondecreasing on $(\theta_0,\theta_1)$.
Therefore, whenever $[\theta-\varepsilon,\theta+\varepsilon] \subset (\theta_0,\theta_1)$ one has
$g(\theta+\varepsilon) \geq g(\theta-\varepsilon)$ which implies
$$
\|x\|_{\theta+\varepsilon} \geq e^{-M(\theta+\varepsilon)}e^{M(\theta-\varepsilon)}
\|x\|_{\theta-\varepsilon} = e^{-2M\varepsilon}  \|x\|_{\theta-\varepsilon}.
$$
Working with $e^{-Ms}\|x\|_s$ instead we obtain
$\|x\|_{\theta+\varepsilon} \leq e^{2M\varepsilon} \|x\|_{\theta-\varepsilon}$.

By density we get $X_{\theta+\varepsilon}=X_{\theta-\varepsilon}$, thus $X_{s} = X_{\theta}$ with
equivalence of norms for $|\theta - s|\leq \varepsilon$, and a result of Stafney \cite[Theorem 1.7]{staf}
implies that $X_0=X_1$ with equivalence of norms.
\end{proof}

Pisier \cite{pisierinter}, motivated by an observation by V. Lafforgue (that certain Banach spaces
called \emph{uniformly curved} do not admit coarse embeddings of expanding graphs), defined the
\emph{$\theta$-euclidean spaces} as those obtained by interpolation of a family of norms on $\C^n$
which are euclidean on a set of positive measure $\theta$, and the {\em $\theta$-hilbertian spaces}
as the ultraproducts of families of $\theta$-euclidian spaces.
He proved that some natural uniformly curved spaces are isomorphic to subspaces of quotients
of $\theta$-hilbertian spaces.
The extrapolation theorem of \cite{pisierlattices} implies that all uniformly convex Banach lattices
are $\theta$-hilbertian and uniformly curved; however the question remains open for uniformly convex
spaces without lattice structure.
Therefore the study of properties of general interpolation scales is also relevant to this context.

\subsection{Isometric rigidity of linear derivations for optimal interpolation pairs}

As we remarked in the Introduction, \cite[Theorem 5.2]{cjmr} proves the estimate
\begin{equation*}\label{difestimate1}
\frac{d}{d\theta} \|x\|_{\theta,1} \sim \|x\|_{\theta,1} + \| \Omega_\theta (x) \|_{\theta,1}.
\end{equation*}
for the real $(\theta, 1)$-method of interpolation.
From this fact, an analogue of Theorem \ref{local} is derived \cite[Theorem 5.16]{cjmr}:
If the maps $\Omega_\theta$ are uniformly bounded for all $|\theta -\theta_0|<\varepsilon$
then $X_0=X_1$.
%
Moreover, \cite[Theorem 5.17]{cjmr} shows that the $(\theta,q)$-method has a stronger stability
property: if $\Omega_\theta$ is bounded for a single inner point $\theta$ then $X_0=X_1$.
A similar result for the complex interpolation method is still unknown in general, but we will
give some partial positive results in this section.
\medskip

We will consider the following subclass of interpolation pairs.

\begin{defin}
An interpolation pair $(X_0,X_1)$ will be called \emph{optimal} if, for every
$0<\theta<1$ and each $x \in X_\theta$, there exists a unique $1$-extremal $f_{x,\theta}$.
\end{defin}

Daher shows in \cite[Proposition 3]{daher} that a regular interpolation pair of reflexive spaces
is optimal when one of the spaces is strictly convex.
The following result was essentially observed by Daher.


\begin{lemma}
Let $(X_0, X_1)$ be an optimal interpolation pair with $X_0$ reflexive.
For all $x \in \Delta$ and $t, z  \in \mathbb{S}$ we have
\begin{enumerate}
\item $\|B_t(x)(z)\|_z = \|x\|_t$,
\item $B_t(x)=B_z\left(B_t(x)(z)\right)$,
\item $B_t(x)'(z) = \Omega_{z} \left(B_t(x)(z)\right)$.
\end{enumerate}
\end{lemma}
\begin{proof}
(1) was proved in Lemma \ref{lemma:extremal}, (2) follows from the uniqueness of the extremals,
since both functions have the same norm and take the value $B_t(x)(z)$ at $z$, and (3) follows
from (2) and $B_t(x)'(z) = \Omega_{z}(x)$.
%
\end{proof}

\begin{lemma}\label{unic}
Let $(X_0,X_1)$ be an optimal interpolation pair.
For all $0<\theta<1$ and $t \in \R$ one has $\Omega_{\theta+it}=\Omega_\theta$.
\end{lemma}
\begin{proof}
Observe that $B_\theta(x)(z-it)= B_{\theta + it}(x)(z)$ since both are extremals for $x$ at $z=\theta+it$.
%
Hence $\Omega_{\theta + it}(x)=B_{\theta+it}(x)'(\theta+it) = B_\theta(x)'(z) =\Omega_\theta(x)$.
\end{proof}

We are ready to obtain some stability results when $\Omega_\theta$ is linear and bounded.
We start with the simplest case $\Omega_\theta=0$.

\begin{prop}\label{nullstabilitypair}
Let $(X_0, X_1)$ be an optimal interpolation pair with $X_0$ reflexive.
Then $\Omega_\theta=0$ for some $0<\theta<1$ if and only if $X_0=X_1$ isometrically.
\end{prop}
\begin{proof}
The if part is well-known and it easily follows from $B_\theta(x)(z) = x$ for $x\in \Delta$.
As for the converse, consider the function $F: \R \to \Sigma$ defined by
$F(t)=B_\theta(x)(\theta+it)$.

This function is constant since $F'(t) = \Omega_{\theta+ it}\left(B_\theta(x)(\theta+ it)\right)=0$.
Thus the analytic function $B_\theta(x)$ is constant on the vertical line through $\theta$, hence
constant on $\mathbb{S}$. In particular
$\|x\|_\theta = \|B_\theta(x)(\theta)\|_\theta = \|B_\theta(x)(z) \|_z = \|x\|_z$ for each $z$.
\end{proof}

Note that this isometric result is new even in the context of K\"othe spaces.



Recall that an operator $T$ acting on a Banach space $X$ is said to be \emph{hermitian} when
$e^{itT}$ is an isometry on $X$ for all $t\in\R$ \cite{bonsallduncan}.


\begin{theorem}\label{unificado}
Let $(X_0,X_1)$ be an interpolation pair of spaces as in Proposition \ref{ri-space}.
Suppose that $(X_0, X_1)$ is optimal and $\Omega_\theta: X_\theta \rightarrow \Sigma$
is linear for some $0<\theta<1$. Then
\begin{enumerate}
\item $\Omega_z(x)=\Omega_\theta(x)$ for all $z\in \mathbb{S}$ and all $x\in \Delta_0$.
\item For every $0<s<1$, the map $x \in \Delta_0\mapsto e^{s\Omega_\theta}x$ induces
an isometry between $X_0$ and $X_s$ which gives $\|x\|_s=\|e^{-s\Omega_\theta} x\|_0$.
\item $\Omega_z$ is an hermitian operator on $X_z$ for all $z\in \mathbb{S}$.
\end{enumerate}
\end{theorem}
\begin{proof}
(1) Since $B_\theta(x)'(\theta+it)= \Omega_{\theta+it}\left(B_\theta(x)(\theta+it)\right) =
\Omega_\theta\left(B_\theta(x)(\theta+it)\right)$ for all $t\in\R$, the function
$t\to B_\theta(x)(\theta+it)$ satisfies the differential equation
\begin{equation}\label{diff}
f^{\prime}(t) = i\Omega_\theta(f(t)).
\end{equation}
Equivalently, $B_\theta(x)$ satisfies the equation $f^{\prime}(z)=\Omega_\theta(f(z))$ for
$z \in \mathbb{S}_\theta = \{z \in \mathbb{S} : Re(z) = \theta\}$.
Since  $B_\theta(x): \mathbb{S} \to \Sigma$ is the unique $1$-extremal and $x\in \Delta_0$,
it is analytic as a map into $\Delta$.
When $\Omega_\theta$ is linear, $\Omega_\theta\circ B_\theta(x): \mathbb{S} \to\Sigma$ is
analytic and takes values in $\Delta$ for $x\in \Delta_0$, and the derivative
$B_\theta(x)': \mathbb{S} \to \Sigma$ is of course analytic.
Since both functions coincide on $S_\theta$, they coincide on $\mathbb{S}$; thus $B_\theta(x)$
solves the equation $f^{\prime}(z)=\Omega_\theta(f(z))$ on $\mathbb{S}$ and we get
\begin{eqnarray*}
\Omega_\theta(x) &=& \Omega_\theta\left(B_z(x)(z)\right) =
   \Omega_\theta\left(B_\theta(B_z(x)(\theta))(z)\right)\\
&=& B_\theta\left(B_z(x)(\theta)\right)'(z) = B_z(x)'(z)=  \Omega_z(x).
\end{eqnarray*}

To prove (2) we need to make sense of the function
$G(t)=e^{-it\Omega_\theta}B_\theta(x)(\theta + it)$ for $x\in \Delta_0$.

Pick $n\in\N$ such that $x\in E_n$.
Since $\Omega_\theta(E_n)\subset E_n$, the iterations $\Omega_\theta^k$ are operators on $E_n$,
so that $G$ is well defined.
Now, since $\Omega_\theta: X_\theta \to \Sigma$ is linear and bounded,
\begin{eqnarray*}
G^{\prime}(t) &= &e^{-it\Omega_\theta} i B_\theta(x)'(\theta+ it)-
   e^{-it\Omega_\theta} i\Omega_\theta\left(B_\theta(x)(\theta + it)\right)\\
&=& e^{-it\Omega_\theta} \Big(i\Omega_\theta \left(B_\theta(x)(\theta+it)\right)
   -i\Omega_\theta \left(B_\theta(x)(\theta+it)\right) \Big)=0.
\end{eqnarray*}
So the function $G(t)$ is constant and equal to $G(0)=x$;
thus $B_\theta(x)(\theta+it) =e^{it\Omega_\theta} x$.
This means that for any $z$ in the vertical line through $\theta$,
$B_\theta(x)(z)=e^{(z-\theta)\Omega_\theta}x$.
Since both functions are analytic on $\mathbb{S}$, it turns out that
$B_\theta(x)(z)=e^{(z-\theta)\Omega_\theta}x$ for all $z\in \mathbb{S}$ and $x\in \Delta_0$.
Since the functions are equal on $\mathbb{S}$, they have the same radial limits a.\ e.\ on the border.

So $B_{\theta}(x)(z) = e^{(z-\theta)\Omega_\theta}x$ for a.\ e.\ $z$ on the border of $\mathbb{S}$.
Thus $1 = \|B_{\theta}(x)(it)\|_0 = \|e^{(it-\theta)\Omega_\theta}x\|_0$ for a.\ e.\ $t \in \mathbb{R}$.
By continuity we have that $\|e^{(it-\theta)\Omega_\theta}x\|_0 = \|x\|_{\theta}$ for every $t \in \mathbb{R}$.
Clearly the same reasoning works for $1 + it$ instead of $it$.

Thus $\|x\|_\theta=\|B_\theta(x)(s)\|_s=\|e^{(s-\theta)\Omega_\theta} x\|_s$ for each $s \in [0,1]$ and $x \in \Delta_0$.
Taking $y = e^{\theta \Omega_{\theta}} x$, we get $\|x\|_0 = \|e^{s \Omega}\Omega_{\theta} x\|_s$ for every $s \in [0, 1]$
and every $x \in \Delta_0$, which is dense in both $X_0$ and $X_s$. Hence the map
$x\to e^{-s\Omega_\theta}x$ extends to an isometry between $X_s$ and $X_0$, and
$\|x\|_s=\|e^{-s\Omega_\theta} x\|_0$.

(3) Since $\|x\|_z=\|B_z(x)(z+it)\|_{z+it}=\|B_z(x)(z+it)\|_z=\|e^{it\Omega_\theta}x\|_z$ and the norm
$\|B_\theta(x)(z)\|_z$ is constant and equal to $\|x\|_\theta$ for $z$ in the vertical line through
$\theta$, we get that $\{e^{it\Omega_\theta}\}_{t\in \R}$ is a group of linear isometries on $X_z$.
\end{proof}

We can compare this result to Theorem \ref{stsplit}, in which the $\Omega_\theta$ trivial implies that
$X_1$ is a weighted version $X_0$ and $\Omega$ is the operator acting as multiplication by $-{\log}\ w$.


\begin{thebibliography}{99}

\bibitem{AAG:92} J.A. Alvarez, T. Alvarez, M. Gonz\'alez,
\emph{The gap between subspaces and perturbation of non-semi-Fredholm operators,}
Bull. Austral. Math. Soc. 45 (1992), 369--376.

\bibitem{accgmLN}
A. Avil\'es, F. Cabello S\'anchez, J.M.F. Castillo, M. Gonz\'alez, Y. Moreno,
\emph{Separably injective Banach spaces,} Lecture Notes in Mathematics, 2132.
Springer, 2016.

\bibitem{bak-newman}
J. Bak, D. J. Newman, \emph{Complex analysis.} Springer, 1982.

\bibitem{bonsallduncan} F. Bonsall, J. Duncan,
\emph{Numerical Ranges of Operators on Normed Spaces and of Elements of Normed Algebras.}
London Math. Soc. Lecture Note Ser. 2. Cambridge University Press, 1971.


\bibitem{BL} J. Bergh, J. L\"ofstr\"om,
\emph {Interpolation spaces. An introduction.} Springer, 1976.



\bibitem{ccgs} F. Cabello, J.M.F. Castillo, S. Goldstein, J. Su\'arez,
\emph{Twisting noncommutative $L_p$-spaces}, Advances in Math. 294 (2016) 454--488.

\bibitem{cck} F. Cabello, J.M.F. Castillo, N. J. Kalton,
\emph{Complex interpolation and twisted twisted Hilbert spaces},
Pacific J. Math. 276 (2015) 287--307.

\bibitem{ccs} F. Cabello, J.M.F. Castillo, J. Su\'arez,
\emph{On strictly singular nonlinear centralizers,} Nonlinear Anal. - TMA 75 (2012) 3313--3321.

\bibitem{calderon} A.-P. Calder\'on,
\emph{Intermediate spaces and interpolation, the complex method,} Studia Math. 24 (1964) 113--190.





\bibitem{casskalt} P.G. Casazza, N.J. Kalton,
\emph{Unconditional bases and unconditional finite-dimensional decompositions in Banach spaces},
Israel J. Math. 95 (1996) 349-373.



\bibitem{cfg} J.M.F. Castillo, V. Ferenczi, M. Gonz\'alez,
\emph{Singular twisted sums generated by complex interpolation},
Trans. Amer. Math. Soc.  369 (2017) 4671--4708.


\bibitem{castgonz} J.M.F. Castillo, M. Gonz\'alez,
\emph{Three-space problems in Banach space theory,} Lecture Notes in Math. 1667. Springer, 1997.


\bibitem{chaatit} F. Chaatit, {\em On uniform homeomorphisms of the unit spheres of certain
Banach lattices}, Pacific J. Math. 168 (1995) 11--31.

\bibitem{Coifman1982} R.R. Coifman, M. Cwikel, R. Rochberg, Y. Sagher, G. Weiss, \emph{A theory of
complex interpolation for families of Banach spaces}, Advances in Math. 43 (1982) 203-229.

\bibitem{correa} W.H.G. Corr\^ea, \emph{Type, cotype and twisted sums induced by complex
interpolation}, J. Funct. Anal. 274 (2018) 797--825.


\bibitem{cjmr} M. Cwikel, B. Jawerth, M. Milman, R. Rochberg,
\emph{Differential estimates and commutators in interpolation theory.} In ``Analysis at Urbana II'',
London Math. Soc. Lecture Note Series 138, (E.R. Berkson, N.T. Peck, and J. Uhl, Eds.), pp. 170--220,
Cambridge Univ. Press, Cambridge, 1989.

\bibitem{cj} M. Cwikel, S. Janson, \emph{Real and complex interpolation methods for finite and infinite families of {B}anach spaces}, Adv. in Math. 66 (1987), no. 3, 234-290.



\bibitem{daher} M. Daher, \emph{Hom\'eomorphismes uniformes entre les sph\`{e}res unit\'e des
espaces d'interpolation,} Canad. Math. Bull. 38 (1995) 286-294.

\bibitem{domast} P. Domanski, M. Masty\l o,
\emph{Characterization of splitting for Fr\'echet-Hilbert spaces via interpolation},
Math. Ann. 339 (2007) 317-340.

\bibitem{duren} P. L. Duren, \emph{Theory of $H^p$ spaces}, Academic Press, New York, 1970.

\bibitem{ferencziunif} V. Ferenczi, {\em A uniformly convex hereditarily indecomposable Banach space},
Israel J. Math. 102 (1997) 199--225.

\bibitem{Hernandez} E. Hernandez,
\emph{Intermediate spaces and the complex method of interpolation for families of Banach spaces},
Ann. Scuola Norm. Sup. Pisa Cl. Sci. 13 (1986), 245--266.




\bibitem{kaltmem} N.J. Kalton, \emph{Nonlinear commutators in interpolation theory,}
Mem. Amer. Math. Soc. 385, 1988.

\bibitem{kaltdiff}
N.J. Kalton, \emph{Differentials of complex interpolation processes for K\"othe function spaces,}
Trans. Amer. Math. Soc. 333 (1992) 479--529.



\bibitem{kalt-mon}
N.J. Kalton, S. Montgomery-Smith, \emph{Interpolation of Banach spaces,}  Chapter 36 in
Handbook of the Geometry of Banach spaces vol. 2, (W.B. Johnson and J. Lindenstrauss eds.),
pp. 1131--1175, North-Holland 2003.

\bibitem{kaltostr} N.J. Kalton, M. Ostrovskii, \emph{Distances between Banach spaces,}
Forum Math. 11 (1999) 17--48.




\bibitem{LaursenNeumann:00} K.B. Laursen, M.M. Neumann,
\emph{An introduction to local spectral theory,} London Math. Soc. Monographs.
New Series, 20. The Clarendon Press, Oxford University Press, New York, 2000.

\bibitem{lindtzaf1} J.~Lindenstrauss, L.~Tzafriri,
\emph{Classical Banach spaces I. Sequence spaces}, Springer, 1977.

\bibitem{lindtzaf-2} J. Lindenstrauss, L. Tzafriri,
\emph{Classical Banach spaces II. Function spaces,} Springer, 1979.

\bibitem{loza} G.Y. Lozanovskii, \emph{On some Banach lattices},
Siberian Math. J. 10 (1969) 419-430.


\bibitem{os} E. Odell, T. Schlumprecht, {\em The distortion problem,}
Acta Math. 173 (1994) 259-281.


\bibitem{pisierlattices} G. Pisier,
\emph{Some applications of the complex interpolation method to Banach lattices},
J. Anal. Math. 35 (1979) 264--281.

\bibitem{pisierinter} G. Pisier, \emph{Complex Interpolation Between Hilbert, Banach and
Operator Spaces}, Mem. Amer. Math. Soc. 978, 2010.

\bibitem{Yanqi-Qiu} Yanqi Qiu,
\emph{On the effect of rearrangement on complex interpolation for families of Banach spaces},
Rev. Mat. Iberoam. 31 (2015) 439--460.




\bibitem{complex} R. Rochberg, \emph{Function theoretic results for complex
interpolation families of Banach spaces}, Trans. Amer. Math. Soc. 284 (1984) 745-758.

\bibitem{roch} R. Rochberg, \emph{Higher order estimates in complex interpolation
theory}, Pacific J. Math. 174 (1996) 247--267.

\bibitem{derivatives} R. Rochberg, G. Weiss, \emph{Derivatives of Analytic Families of Banach
Spaces}, Annals of Math. 118 (1983) 315--347.


\bibitem{yanqui} Q. Yanqui, \emph{On the effect of rearrangement on complex interpolation for families of {B}anach spaces}, Rev. Mat. Iberoam. 31 (2015), 439-460.

\bibitem{staf} J. Stafney, \emph{Analytic interpolation of certain multiplier spaces,}
Pacific J. Math 32 (1970) 241-248.

\bibitem{zaf} M. Zafran, \emph{Spectral theory and interpolation of operators,}
J. Funct. Anal. 36 (1980) 185-204.


\end{thebibliography}
\end{document}